\documentclass{amsart}

\usepackage{amsfonts,amssymb,verbatim,amsmath,amsthm,latexsym,textcomp,amscd, hyperref}
\usepackage{epsfig}
\usepackage{amsmath,amsthm,graphics,float}
\usepackage{graphicx}
\usepackage{xcolor}
\usepackage{url}
\usepackage{thm-restate}

\begin{document}

\newtheorem{theorem}{Theorem}[section]
\newtheorem{prop}[theorem]{Proposition}
\newtheorem{lemma}[theorem]{Lemma}
\newtheorem{cor}[theorem]{Corollary}
\newtheorem{cond}[theorem]{Condition}
\newtheorem{ing}[theorem]{Ingredients}
\newtheorem{conj}[theorem]{Conjecture}
\newtheorem{claim}[theorem]{Claim}
\newtheorem{constr}[theorem]{Construction}
\newtheorem{rem}[theorem]{Remark}
\newtheorem{defn}[theorem]{Definition}
\newtheorem{eg}[theorem]{Example}
\newtheorem{rmk}[theorem]{Remark}

\newtheorem*{theorem*}{Theorem}
\newtheorem*{modf}{Modification for arbitrary $n$}
\newtheorem{qn}[theorem]{Question}
\newtheorem{condn}[theorem]{Condition}
\newtheorem*{BGIT}{Bounded Geodesic Image Theorem}
\newtheorem*{BI}{Behrstock Inequality}
\newtheorem*{QCH}{Wise's Quasiconvex Hierarchy Theorem}

\newcommand{\map}{\rightarrow}
\newcommand{\boundary}{\partial}
\newcommand{\C}{{\mathbb C}}
\newcommand{\integers}{{\mathbb Z}}
\newcommand{\G}{{\Gamma}}
\newcommand{\natls}{{\mathbb N}}
\newcommand{\ratls}{{\mathbb Q}}
\newcommand{\reals}{{\mathbb R}}
\newcommand{\proj}{{\mathbb P}}
\newcommand{\lhp}{{\mathbb L}}
\newcommand{\tr}{{\operatorname{Tread}}}
\newcommand{\rs}{{\operatorname{Riser}}}
\newcommand{\tube}{{\mathbb T}}
\newcommand{\cusp}{{\mathbb P}}
\newcommand\AAA{{\mathcal A}}
\newcommand\BB{{\mathcal B}}
\newcommand\CC{{\mathcal C}}
\newcommand\ccd{{{\mathcal C}_\Delta}}
\newcommand\DD{{\mathcal D}}
\newcommand\EE{{\mathcal E}}
\newcommand\FF{{\mathcal F}}

\newcommand\GG{{\mathcal G}}
\newcommand\GGs{{{\mathcal G}^*}}
\newcommand\HH{{\mathcal H}}
\newcommand\II{{\mathcal I}}
\newcommand\JJ{{\mathcal J}}
\newcommand\KK{{\mathcal K}}
\newcommand\LL{{\mathcal L}}
\newcommand\MM{{\mathcal M}}
\newcommand\NN{{\mathcal N}}
\newcommand\OO{{\mathcal O}}
\newcommand\PP{{\mathcal P}}
\newcommand\QQ{{\mathcal Q}}
\newcommand\RR{{\mathcal R}}
\newcommand\SSS{{\mathcal S}}

\newcommand\TT{{\mathcal T}}
\newcommand\ttt{{\mathcal T}_T}
\newcommand\tT{{\widetilde T}}
\newcommand\UU{{\mathcal U}}
\newcommand\VV{{\mathcal V}}
\newcommand\WW{{\mathcal W}}
\newcommand\XX{{\mathcal X}}
\newcommand\YY{{\mathcal Y}}
\newcommand\ZZ{{\mathcal Z}}
\newcommand\CH{{\CC\HH}}
\newcommand\TC{{\TT\CC}}
\newcommand\EXH{{ \EE (X, \HH )}}
\newcommand\GXH{{ \GG (X, \HH )}}
\newcommand\GYH{{ \GG (Y, \HH )}}
\newcommand\PEX{{\PP\EE  (X, \HH , \GG , \LL )}}
\newcommand\MF{{\MM\FF}}
\newcommand\PMF{{\PP\kern-2pt\MM\FF}}
\newcommand\ML{{\MM\LL}}
\newcommand\mr{{\RR_\MM}}
\newcommand\tmr{{\til{\RR_\MM}}}
\newcommand\PML{{\PP\kern-2pt\MM\LL}}
\newcommand\GL{{\GG\LL}}
\newcommand\Pol{{\mathcal P}}
\newcommand\half{{\textstyle{\frac12}}}
\newcommand\Half{{\frac12}}
\newcommand\Mod{\operatorname{Mod}}
\newcommand\Area{\operatorname{Area}}
\newcommand\ep{\epsilon}
\newcommand\hhat{\widehat}
\newcommand\Proj{{\mathbf P}}
\newcommand\U{{\mathbf U}}
 \newcommand\Hyp{{\mathbf H}}
\newcommand\D{{\mathbf D}}
\newcommand\Z{{\mathbb Z}}
\newcommand\K{{\mathbb K}}
\newcommand\R{{\mathbb R}}
\newcommand\bN{\mathbb{N}}
\newcommand\s{{\Sigma}}
\renewcommand\P{{\mathbb P}}
\newcommand\tga{\til{\gamma}}
\newcommand\Q{{\mathbb Q}}
\newcommand\E{{\mathbb E}}
\newcommand\til{\widetilde}
\newcommand\length{\operatorname{length}}
\newcommand\BU{\operatorname{BU}}
\newcommand\gesim{\succ}
\newcommand\lesim{\prec}
\newcommand\simle{\lesim}
\newcommand\simge{\gesim}
\newcommand{\simmult}{\asymp}
\newcommand{\simadd}{\mathrel{\overset{\text{\tiny $+$}}{\sim}}}
\newcommand{\ssm}{\setminus}
\newcommand{\diam}{\operatorname{diam}}
\newcommand{\pair}[1]{\langle #1\rangle}
\newcommand{\T}{{\mathbf T}}
\newcommand{\inj}{\operatorname{inj}}
\newcommand{\pleat}{\operatorname{\mathbf{pleat}}}
\newcommand{\short}{\operatorname{\mathbf{short}}}
\newcommand{\vertices}{\operatorname{vert}}
\newcommand{\collar}{\operatorname{\mathbf{collar}}}
\newcommand{\bcollar}{\operatorname{\overline{\mathbf{collar}}}}
\newcommand{\I}{{\mathbf I}}
\newcommand{\tprec}{\prec_t}
\newcommand{\fprec}{\prec_f}
\newcommand{\bprec}{\prec_b}
\newcommand{\pprec}{\prec_p}
\newcommand{\ppreceq}{\preceq_p}
\newcommand{\sprec}{\prec_s}
\newcommand{\cpreceq}{\preceq_c}
\newcommand{\cprec}{\prec_c}
\newcommand{\topprec}{\prec_{\rm top}}
\newcommand{\Topprec}{\prec_{\rm TOP}}
\newcommand{\fsub}{\mathrel{\scriptstyle\searrow}}
\newcommand{\bsub}{\mathrel{\scriptstyle\swarrow}}
\newcommand{\fsubd}{\mathrel{{\scriptstyle\searrow}\kern-1ex^d\kern0.5ex}}
\newcommand{\bsubd}{\mathrel{{\scriptstyle\swarrow}\kern-1.6ex^d\kern0.8ex}}
\newcommand{\fsubeq}{\mathrel{\raise-.7ex\hbox{$\overset{\searrow}{=}$}}}
\newcommand{\bsubeq}{\mathrel{\raise-.7ex\hbox{$\overset{\swarrow}{=}$}}}
\newcommand{\tw}{\operatorname{tw}}
\newcommand{\base}{\operatorname{base}}
\newcommand{\trans}{\operatorname{trans}}
\newcommand{\rest}{|_}
\newcommand{\bbar}{\overline}
\newcommand{\UML}{\operatorname{\UU\MM\LL}}
\renewcommand{\d}{\operatorname{diam}}
\newcommand{\hs}{{\operatorname{hs}}}
\newcommand{\EL}{\mathcal{EL}}
\newcommand{\tsum}{\sideset{}{'}\sum}
\newcommand{\tsh}[1]{\left\{\kern-.9ex\left\{#1\right\}\kern-.9ex\right\}}
\newcommand{\Tsh}[2]{\tsh{#2}_{#1}}
\newcommand{\qeq}{\mathrel{\approx}}
\newcommand{\Qeq}[1]{\mathrel{\approx_{#1}}}
\newcommand{\qle}{\lesssim}
\newcommand{\Qle}[1]{\mathrel{\lesssim_{#1}}}
\newcommand{\simp}{\operatorname{simp}}
\newcommand{\vsucc}{\operatorname{succ}}
\newcommand{\vpred}{\operatorname{pred}}
\newcommand\sleft{_{\text{left}}}
\newcommand\sright{_{\text{right}}}
\newcommand\sbtop{_{\text{top}}}
\newcommand\sbot{_{\text{bot}}}
\newcommand\dpe{d_{pel}}
\newcommand\de{d_{e}}
\newcommand\srr{_{\mathbf r}}
\newcommand\geod{\operatorname{\mathbf g}}
\newcommand\mtorus[1]{\boundary U(#1)}
\newcommand\A{\mathbf A}
\newcommand\Aleft[1]{\A\sleft(#1)}
\newcommand\Aright[1]{\A\sright(#1)}
\newcommand\Atop[1]{\A\sbtop(#1)}
\newcommand\Abot[1]{\A\sbot(#1)}
\newcommand\boundvert{{\boundary_{||}}}
\newcommand\storus[1]{U(#1)}
\newcommand\Momega{\omega_M}
\newcommand\nomega{\omega_\nu}
\newcommand\twist{\operatorname{tw}}
\newcommand\SSSS{{\til{\mathcal S}}}
\newcommand\modl{M_\nu}
\newcommand\MT{{\mathbb T}}
\newcommand\dw{{d_{weld}}}
\newcommand\dt{{d_{te}}}
\newcommand\Teich{{\operatorname{Teich}}}
\renewcommand{\Re}{\operatorname{Re}}
\renewcommand{\Im}{\operatorname{Im}}
\newcommand{\mc}{\mathcal}
\newcommand{\ccs}{{\CC(S)}}
\newcommand{\mtdw}{{(\til{M_T},\dw)}}
\newcommand{\tmtdw}{{(\til{M_T},\dw)}}
\newcommand{\tmldw}{{(\til{M_l},\dw)}}
\newcommand{\sxy}{{\Sigma(x,y)}}
\newcommand{\ts}{{\til{\Sigma}}}
\newcommand{\tsxy}{{\til{\Sigma}(x,y)}}
\newcommand{\tmtdt}{{(\til{M_T},\dt)}}
\newcommand{\tmldt}{{(\til{M_l},\dt)}}
\newcommand{\trvw}{{\tr_{vw}}}
\newcommand{\ttrvw}{{\til{\tr_{vw}}}}
\newcommand{\but}{{\BU(T)}}
\newcommand{\ilkv}{{i(lk(v))}}
\newcommand{\pslc}{{\mathrm{PSL}_2 (\mathbb{C})}}
\newcommand{\tttt}{{\til{\ttt}}}
\newcommand{\bcomment}[1]{\textcolor{blue}{#1}}

\newcommand{\Star}{\operatorname{star}}
\newcommand{\jfmchange}[1]{{\color{purple}{#1}}}

\newcommand{\defstyle}[1]{\textbf{#1}}
\newcommand{\emphstyle}[1]{\emph{#1}}

\title{Cubulating Drilled bundles over graphs}

\author{Mahan Mj}

\address{School of Mathematics, Tata Institute of Fundamental Research, Mumbai-40005, India}
\email{mahan@math.tifr.res.in}
\email{mahan.mj@gmail.com}

\author{Biswajit Nag}

\address{School of Mathematics, Tata Institute of Fundamental Research, Mumbai-40005, India}
\email{biswajit@math.tifr.res.in}


\date{\today}

\thanks{Both authors were  supported by  the Department of Atomic Energy, Government of India, under project no.12-R\&D-TFR-5.01-0500 as also  by an endowment of the Infosys Foundation.
	MM was also supported in part by a DST JC Bose Fellowship. MM was  supported in part by  the Institut Henri Poincare  (UAR 839 CNRS-Sorbonne Universite), LabEx CARMIN, ANR-10-LABX-59-01, during his participation in the trimester program "Groups acting on fractals, Hyperbolicity and Self-similarity", April-June 2022.} 

\begin{abstract} 
	We start with a Gromov-hyperbolic surface bundle $E$ over a graph, and drill out essential simple closed curves from fibers
	to obtain a drilled bundle $F$. We prove that for such  drilled bundles $F$, the fundamental group $\pi_1(F)$ is  relatively hyperbolic with
	$(\Z\oplus \Z)$ peripheral groups.  Combining the relative hyperbolicity of $\pi_1(F)$ thus obtained with a theorem of Wise, we establish 
	virtually special cubulability of $\pi_1(F)$ provided that  the maximal undrilled subbundles of $F$ are cubulable.
\end{abstract}

\maketitle

\tableofcontents

\section{Introduction}\label{sec-intro} In \cite{mmscubulating}, Manning-Mj-Sageev investigated the following question:

\begin{qn}\label{qn-mms}
Let $1\to H \to G \to F_n \to 1$ be an exact sequence of hyperbolic groups with $F_n$ the free group on $n$ generators. Is $G$ virtually special cubulable?
\end{qn}
 They   gave sufficient conditions guaranteeing an affirmative answer. Surface-by-free groups
 as in Question~\ref{qn-mms} naturally arise as fundamental groups of surface bundles $E$ over graphs $\GG$. The main contribution of \cite{mmscubulating} was an explicit  construction of a codimension one quasiconvex subgroup of $G$ along which $G$ splits. Wise's quasiconvex hierarchy theorem \cite{wise-hier} then furnishes cubulability of such groups. When $\GG$ is a circle, Question~\ref{qn-mms}
 has an affirmative answer thanks to Kahn-Markovic's work on the surface subgroup problem  \cite{km-surf}, Bergeron-Wise's cubulability result
  \cite{bw}, and Agol's theorem \cite{agol-vh}. In this case,
 $E$ is a hyperbolic 3-manifold $M$ fibering over the circle.
 The existence of an embedded quasiconvex  surface in such an $M$ is thus not guaranteed when the first betti number of $M$ is one. On the other hand, for
 hyperbolic 3-manifolds with toral boundary components, the construction of embedded geometrically finite surfaces is much easier. In particular, if one drills out a simple closed curve $\sigma$ from a fiber $S$ of a fibered manifold $M$ as above, then $M \setminus \sigma$ admits a complete hyperbolic structure
 by Thurston's theorem \cite[Ch. 15]{kapovich-book}, and each component of
 $S \setminus \sigma$ is geometrically finite. 
 In this paper, we shall adopt the point of view that a surface bundle  $E$ 
 over  a graph $\GG$  generalizes  hyperbolic 3-manifolds  fibering over the circle.
  
Here, the main objects of study will be drilled surface bundles over graphs, where drilling corresponds to removing open neighborhoods of simple closed curves in fibers. More precisely: 

\begin{defn}\label{def-drillbdl}
 Let $\GG$ be a connected graph, thought of as a $1$--complex, and consider a bundle $\Pi: E\to \GG$ with fiber $S$ a surface. We refer to $\Pi: E\to \GG$ as a \emph{surface bundle over the graph} $\GG$.
 The fiber $\Pi^{-1}(x)$ over $x \in \GG$ will be denoted by $S_x$. If $x$ is a vertex of $\GG$, $S_x$ will be called a \emph{singular fiber}, else it will be called a \emph{regular fiber}.
 
 Let $\sigma_i \subset {S_{x_i}}$ be a finite collection of essential simple closed curves in  regular fibers ${S_{x_i}}$, so that for each regular
 fiber $S_x$, the collection of simple closed curves contained in $S_x$ are disjoint. Let $\{{N_\ep(\sigma_i)}\}$ be a collection of 
 small open tubular neighborhoods, missing singular fibers. We assume that 
 \begin{enumerate}
\item the closures  $\{\bbar{N_\ep(\sigma_i)}\}$ are disjoint,
\item for $i\neq j$, $\sigma_i, \sigma_j$ are not freely homotopic in $E$, nor is there a nontrivial free homotopy between  $\sigma_i$ and itself. This can equivalently be described as follows. Let $F=(E \setminus \cup_i N_\ep(\sigma_i))$.
Then any $\pi_1-$injective smooth map $(S^1 \times [0,1], S^1\times \{0,1\}) \to (F, \, \cup_i \partial N_\ep(\sigma_i)))$ is homotopic, rel. boundary, into $\partial N_\ep(\sigma_i)$ for some  $i$.
 \end{enumerate}
 The complement  $F=(E \setminus \cup_i N_\ep(\sigma_i))$  will be referred to as a  \emph{drilled surface bundle over a graph}. 
 
\end{defn}
Each torus $\{\partial{N_\ep(\sigma_i)}\}$ will be referred to as a \emph{boundary torus} of $F$, and denoted as $T_i$. The union  $\cup_i T_i$ will be called the \emph{boundary} of $F$. The surfaces ${S_{x_i}}$ (containing some $\sigma_i$) will be called \emph{drilled surfaces} and the curves $\sigma_i$ will be called \emph{drilled curves}. The points ${x_i} \in \GG$ will be called \emph{drilled points}. 
 An  edge $e$ containing a drilled point will be called a \emph{drilled edge}.
 A  $\pi_1-$injective smooth map $(S^1 \times [0,1], S^1\times \{0,1\}) \to (F, \, \cup_i \partial N_\ep(\sigma_i)))$ will be referred to as an \emph{essential annulus.} The second condition in Definition~\ref{def-drillbdl} thus says that
 $F$ has no essential annuli.
 
We shall be particularly interested in the case that $\GG$ is finite, and
$\Gamma=\pi_1(E)$ is Gromov hyperbolic. In this special case, we note the further restriction on drilled curves that follows from 
the second condition in Definition~\ref{def-drillbdl}. Let $E_S$ be the cover of $E$ corresponding to $\pi_1(S)$. Then we have:

\begin{condn}\label{cond-hyp}
Let $\theta_1, \theta_2$ be any two distinct elevations of the drilled curves $\sigma_i \subset E$ to $E_S$. Then $\theta_1, \theta_2$ are not  freely homotopic in $E_S$ in the complement of other elevated curves. Equivalently, $E_S$ has
\emph{no essential annuli}.
\end{condn}

Condition~\ref{cond-hyp} follows from item 2  in Definition~\ref{def-drillbdl}, viz.\  that $E$ has no essential annuli. However, we make this explicit as much of this paper will involve $E_S$. 

The bulk of this paper goes into establishing 
strong relative hyperbolicity of fundamental groups of drilled bundles (see Theorem~\ref{thm-relhyp}):

\begin{theorem}\label{thm-relhyp-intro} Let $E$ be a surface bundle over a graph $\GG$ such that 
 $\Gamma=\pi_1(E)$ is hyperbolic. Let $F$ be a drilled surface bundle over  $\GG$ obtained by drilling $E$.
  Then $G(=\pi_1(F))$ is strongly hyperbolic relative to the collection of peripheral subgroups $\{P_i = \pi_1(\partial N_\ep(\sigma_i))\}$.
\end{theorem}
 
 The drilling operation  in this paper and Theorem~\ref{thm-relhyp-intro} is  motivated, in part, by the drilling of simple closed geodesics from closed hyperbolic 3-manifolds \cite[Chapter 4]{thurstonnotes}. We also refer to  \cite{beleg-hruska,beleg,ghmosw} for related, but different, drilling 
 constructions: in the first two, totally geodesic codimension 2 submanifolds are drilled out of CAT(-1) manifolds, and in the last, closed geodesics are drilled out of hyperbolic $PD(3)$ groups. 
 
 With Theorem~\ref{thm-relhyp-intro} in place, we can cut $F$ along the components
 $K$ of $S\setminus \cup_i \sigma_i$. It turns out that the fundamental group
 $\pi_1(K)$ of each such component $K$ is relatively quasiconvex in $G$
 (see Proposition~\ref{prop-relqc}). We can now apply Wise's relatively hyperbolic version of the quasiconvex hierarchy theorem \cite[Theorem 15.1]{wise-hier}. To state
 the main theorem of this paper, a bit more terminology needs to be set up.
 Let $\KK_1, \cdots, \KK_n$ denote maximal subgraphs of $\GG$ that contain 
 no points $x$ such that the fiber $S_x$ is drilled. We refer to the 
 $\KK_i$'s as undrilled components of $\GG$. The restrictions 
 $E(\KK_1), \cdots, E(\KK_n)$ will be termed  undrilled constituents of $F$
 (note that each such $E(\KK_n)$ is naturally contained in $F$). The main theorem
 of this paper is the following (see Theorem~\ref{thm-cub}).
 
 \begin{theorem}\label{thm-cub-intro}
 If for each  undrilled constituent $E(\KK)$ of $F$, $\pi_1(E(\KK))$ is
 cubulable virtually special, then so is $G=\pi_1(F)$.
 \end{theorem}
 
 Thus, if Question~\ref{qn-mms} has a positive answer (the undrilled case), then
 for the drilled groups $G$, cubulability follows. However, even in the absence
 of a definitive answer to  Question~\ref{qn-mms}, Theorem~\ref{thm-cub-intro}
 furnishes a number of examples, as given below:
 \begin{enumerate}
 \item when each edge of $\GG$ contains an $x$ such that $S_x$ is drilled (Example~\ref{eg-contratcible}), 
 \item when  undrilled components of $\GG$ are either contractible or homotopy equivalent to a circle (Example~\ref{eg-3mfld}),
 \item  when  undrilled constituents of $F$ satisfy the sufficient conditions in \cite{mmscubulating} (Example~\ref{eg-mms}). 
 \end{enumerate}
 
 Finally, we use Theorem~\ref{thm-cub-intro} in conjunction with Kielak's theorem \cite{kielak} to deduce that the cubulable virtually special groups 
 virtually algebraically fiber (see Theorem~\ref{thm-vfib}). 
 
 \begin{theorem}\label{thm-vfib-intro}
 	Let $F$ be a drilled surface bundle over a finite graph $\GG$ satisfying the hypotheses of Theorem~\ref{thm-cub-intro}. Then $G=\pi_1(F)$ virtually algebraically fibers. 
 \end{theorem} 
 
 To prove Theorem~\ref{thm-vfib-intro}, we  first establish that for any drilled surface bundle $F$ over a graph,
 $\pi_1(F)$ has vanishing first $l^2$ betti number (Proposition~\ref{prop-l2}). 
 This is done using work of Lott-Lueck \cite{lott-luecke} and Fernos-Valette \cite{fv-hha}
 
 \section{Bundles, drilled bundles, graphs of groups} \label{sec-prel-bdl}
Consider the exact sequence 
\begin{equation}\label{eq-ses}
1 \to H \to \G \to Q \to 1,
\end{equation}
 where $H= \pi_1(S)$, and $Q = \pi_1(\GG)$.
 Note that $\G$ has a graph of groups structure, where each edge and vertex group equals $H$
 and all edge-to-vertex maps are isomorphisms. Also, $E$ has a 
  graph of spaces structure, where each edge and vertex space equals $S$, and all edge-to-vertex maps are homeomorphisms.
 A description of $\pi_1E$ in general may then be given as follows (see \cite[Section 2]{mmscubulating} for instance).  Choose a 
 maximal tree $T\subset \GG$.  Assume, without loss of generality, that for any edge $e\subset T$, the gluing maps $f_e^-$ are  identity maps on $S$.  Let $e_1,\ldots,e_n$ be the edges in $\GG\setminus T$.  For each $i\in \{1,\ldots,n\}$, denote $f_i = f_{e_i}^+$. Also, set $\phi_i = (f_i)_*: \pi_1S\to \pi_1S$.  Then $\pi_1E$ is given by:
	\[ \pi_1E \cong \langle \pi_1S,t_1,\ldots,t_n\mid t_i^{-1} st_i=\phi_i(s),\forall s\in \pi_1S, i\in\{1,\ldots,n\}\rangle.\]\\

	\noindent {\bf Hyperbolic $\G$:}$  $\\	
	It was shown by Hamenstadt \cite{hamen} (see also \cite{kl-coco,mahan-sardar}) that,
	in the exact sequence~\eqref{eq-ses}, $\Gamma$ is hyperbolic if and only if $Q$ is a convex cocompact subgroup of the Mapping Class Group in the sense of Farb-Mosher \cite{farb-coco}. A useful fact that we will need is the following Scott-Swarup type theorem \cite{scottswar} due to
	Dowdall-Kent-Leininger \cite[Theorem 1.3]{kld-coco} (see \cite{mahan-rafi} for a different proof of the same result).
	
	\begin{theorem}\label{kld}
	Let
		\[
		1 \rightarrow H \rightarrow \Gamma \rightarrow Q \rightarrow 1
		\]
	be an exact sequence as above, so that $\Gamma$ is hyperbolic, and $H=\pi_1(S)$.
			Then any finitely generated infinite index subgroup of \(H\) is quasiconvex in \(\Gamma\).
		\end{theorem}

		\subsection*{A graph of groups structure on $G=\pi_1(F)$:} $  $\\	
	We now describe a  standard graph of groups description for $\pi_1(F)$. We shall denote
	$G:=\pi_1(F)$. 
	 The  graph in question will be a kind of a dual graph $\GGs$ whose vertex set $\VV(\GG^*)$  will be
	 \begin{enumerate}
	 \item  edges $e$ of $\GG$, and
	 \item vertices $v$ of $\GG$.
	 \end{enumerate} 
	 The edge set $\EE(\GGs)$
	 of $\GGs$ is given by incidences between  edges $e$ and vertices $v$ of $\GG$, i.e.\ if
	the terminal vertex $e^+$ of $e\in \GG$  is $v$, then we introduce an edge in $\GGs$ between $e$ and $v$. Similarly for $e^-$.
	 
	  Then $G(=\pi_1(F))$ has a graph of 
	groups structure, where
	\begin{enumerate}
		\item the underlying graph is $\GGs$,
		\item the edge groups are all equal to $H=\pi_1(S)$,
	\item the vertex groups $G_e$ (indexed by $e \in \GG$) are given as follows. Since the bundle over the interior of ${e} \subset \GG$ is $S \times (0,1)$, $G_e$ equals $\pi_1(M_e)$, where $M_e$ is a possibly drilled copy of $S \times [0,1]$.
% if the closure $\bar{e} \subset \GG$ is an embedded copy of a circle, 
%	then $G_e$ equals $\pi_1(M_e)$, where $M_e$ is a possibly drilled copy of a 3-manifold fibering over the circle. \end{itemize}
\item in particular,  the edge-group to vertex-group maps are injective.
	\end{enumerate} 
	
A convenient way to think of the graph of spaces associated to the above graph of groups description is as follows. For any edge $e \in \GG$, the corresponding vertex space over $e\in \GG^*$ is $M_e$, the possibly drilled copy of $S \times [0,1]$. For any vertex $v \in \GG$, the corresponding vertex space over $v\in \GG^*$ is $S \times \Star(v)$, where $\Star(v)$ denotes a small closed neighborhood of $v$ in $\GG$. The edge spaces are all given by $S$ and edge-to-vertex space inclusions are given by the inclusions of $S$ into either a 
drilled copy of $S \times [0,1]$ or a copy of $S \times \Star(v)$.

Let $\PP_i = \pi_1(T_i)= \Z+\Z$. Then there exists $e$ such that $\PP_i \subset 
G_e=\pi_1(M_e)$. Whenever we are dealing with $G$ equipped with the above graph of groups structure, we shall re-index the collections $\{T_i\}$ and $\{\PP_i\}$ so that the collection of tori contained in $M_e$ are indexed as $\{T_e^j\}$. Similarly, at the level of fundamental groups, the collection of $\{\PP_i\}$ contained in $G_e$ 
are indexed as $\{\PP_e^j\}$. 

\begin{lemma}\label{lem-rh-bb}
Suppose that $\Gamma$ is hyperbolic. Let $\sigma_i$ denote a finite collection of curves that are drilled from $E$ to obtain $F$. Let $\QQ_i \subset \Gamma$ denote the (conjugacy class of the) cyclic subgroup corresponding to $\sigma_i$. 
Then  $\Gamma$ is strongly hyperbolic relative to the collection $\{\QQ_i\}$.
\end{lemma}

\begin{proof}
Since $\sigma_i$'s are simple closed curves, and since $\Gamma$ is hyperbolic they denote primitive elements of $\Gamma$. Hence the collection 
$\{\QQ_i\}$ denotes a malnormal quasiconvex family of subgroups. By a theorem of Bowditch \cite[Theorem 7.11]{bowditch-relhyp}, $\Gamma$ is strongly hyperbolic relative to the collection $\{\QQ_i\}$.
\end{proof}

\section{Relative Hyperbolicity} We refer the reader to \cite{farb-relhyp,bowditch-relhyp} for generalities on relative hyperbolicity.
This section and the next is devoted to proving  the following:

\begin{theorem}\label{thm-relhyp}
If $\Gamma$ is hyperbolic, then $G$ is (strongly) hyperbolic relative to the collection 
$\{\PP_i\}$.
\end{theorem}

Theorem~\ref{thm-relhyp} says roughly that the result of drilling a hyperbolic 3-complex fibering over a graph gives a relatively hyperbolic 3-complex. This result is along the lines of earlier work of Belegradek-Hruska
 \cite{beleg-hruska} (see also \cite{beleg}), who prove similar results in a manifold context.
 The proof of Theorem~\ref{thm-relhyp} occupies the rest of this section and the next.

 \subsection{Drilling 3-manifolds}\label{sec-drill3} Let $M$ be either $S \times I$
 or a hyperbolic 3-manifold fibering over the circle with fiber $S$. Thus, $M$ fibers over a compact 1-manifold (possibly with boundary) with fiber $S$.
 We proceed to drill simple closed curves $ \sigma_i$ in some finite collection of fibers.
 A fiber $S'$ will be referred to as \emph{undrilled} if $S'$ does not intersect any
 of the neighborhoods $N_\ep (\sigma_i)$ of the drilled curves. Let $M_r$ denote the drilled manifold. Henceforth, in this paper, $M_r$ will be equipped with  
 \begin{enumerate}
 	\item a complete hyperbolic structure  of finite volume when $M$ fibers over the circle,
 	\item a geometrically finite hyperbolic structure with convex boundary otherwise.
 \end{enumerate} 
 In the latter case, we shall think of $M_r$ as the quotient of the convex hull of
 a geometrically finite representation $\rho: \pi_1(M_r) \to \pslc$ such that the 
 parabolics of $\rho ( \pi_1(M_r))$ correspond precisely to the boundary tori of $M_r$.
Since no two of the drilled curves are freely homotopic in $M$ by Definition~\ref{def-drillbdl}, it follows that $M_r$ is atoroidal.
Since $M_r$ is an atoroidal Haken manifold, the existence of such hyperbolic structures in both cases follows from Thurston's theorem \cite[Ch. 15]{kapovich-book}.

\begin{defn}\label{def-preferred}
Hyperbolic structures as above on $M_r$ will be referred to as \emph{preferred hyperbolic structures}. If $S_x$ is a drilled fiber in $M$, where $\sigma \subset S_x$ is a simple (multi-)curve that is drilled, then
any connected component of $S_x \setminus \sigma$ will be called a 
\emph{preferred subsurface} of a drilled fiber.
\end{defn}

Note that for a preferred hyperbolic structure on $M_r$,  $M_r$ is finite volume without boundary if $M$ fibers over the circle. Else, the boundary consists of
 two copies of $S$,  corresponding to the (singular) fibers over $0, 1 \in I$. Further,
$M_r$ has finitely many rank two cusps (corresponding to the toral boundaries of regular neighborhoods of drilled curves).

 We start with the following observation:
 
 \begin{lemma}\label{lem-geof}
 Let $M, M_r$ be as above where the set of drilled curves correspond to
a finite non-empty collection $\sigma_i \in S_i$   of simple closed curves on fibers $S_i$. Any undrilled fiber $S'$ is geometrically finite in $M_r$. Any preferred 
subsurface of a drilled fiber is also geometrically finite in $M_r$.
 \end{lemma} 

\begin{proof}
This  follows from the covering theorem \cite[Theorem 9.2.2]{thurstonnotes} \cite{canary-cover}. 
\end{proof}

\begin{comment}
	content...

Refer back to sec 3.1--restate lemma 3.2 in terms of thurston's hyperbolization of drilled atoms as they are atoroidal Haken (say this is true due to the condition
on drilled curves not being isotopic). 

\begin{defn}\label{def-fd}
3-manifolds, where the drilled curves satisfy the hypothesis of Lemma~\ref{lem-geof} are termed \emph{fully drilled 3-manifolds}.
\end{defn}
\end{comment}

 \noindent {\bf Essential Annuli in atoms:}\\ We now describe the essential annuli in atoms. 
 \begin{defn}\label{def-essential}
 An \emph{ immersed essential annulus} or simply an \emph{essential  annulus} $A$ in a 3-manifold $M$ with non-empty boundary is an immersion $i: (A, \partial A) \to (M, \partial M)$ such that 
 \begin{enumerate}
 \item $i_*: \pi_1(A) \to \pi_1(M)$ is injective
 \item $i$ is not homotopic rel. boundary into $\partial M$. 
 \end{enumerate}
 When $M$ is a 3-manifold  fibering over the circle with a distinguished singular fiber $S_x$, an \emph{ immersed essential annulus} or simply an \emph{essential  annulus} $A$
 is an immersion $i: (A, \partial A) \to (M, S_x)$
 such that 
 \begin{enumerate}
 	\item $i_*: \pi_1(A) \to \pi_1(M)$ is injective
 	\item $i$ is not homotopic rel. boundary into $S_x$. 
 \end{enumerate}
 \end{defn}
 Henceforth, we shall refer to immersed essential  annuli simply as essential  annuli.
 In the 3-manifold literature, essential  annuli often refer to embedded essential  annuli in $ (M, \partial M)$. However, since we shall not have any special use for embedded annuli, we shall use the terminology from Definition~\ref{def-essential} in this paper. The usage is consistent with the more group-theoretic notion in \cite{BF} (see Definition~\ref{def-hyphallway} below).
 Note that essential  annuli $A$ in fibered $M$ are allowed to intersect $S_x$ in the interior of $A$ as well, i.e.\ $A$ is allowed to `wrap around transverse to the fibers' of $M$ multiple times so long as $\partial A$  maps to $S_x$.
 Essential annuli in undrilled atoms are given by the following \emph{up to homotopy}:
 \begin{enumerate}
 \item If $M=S \times I$, then essential annuli are of the form $\sigma \times I$, where
 $\sigma$ is an essential, possibly immersed curve in $S$. After homotopy, we may assume that  $\sigma$ is a geodesic in some auxiliary hyperbolic structure on $S$.
 \item If $M$ fibers over the circle, let $M_\Z$ denote the cover of $M$ corresponding to $\pi_1(S)$, so that $M_\Z$ is homeomorphic to $S \times \R$. Then essential annuli in $M_\Z$  are concatenations of annuli $A_i \subset S \times [i, i+1]$ as in item (1). Further, 
 $\sigma_{i+1}=A_i \cap S \times \{ i+1\}=A_{i+1} \cap S \times \{ i+1\}$, so that $A_i, A_{i+1}$ may be concatenated along the essential, possibly immersed curve 
 $\sigma_{i+1}$.
 \end{enumerate}
 
 Essential annuli in drilled atoms are a bit more involved.
 Let $M=S \times I$, and $M_r$ be a drilled atom obtained from $M$. Let $\sigma_1, \cdots, \sigma_m$ denote the drilled curves on $S$. Realize $\sigma_i$'s by geodesics in an auxiliary hyperbolic structure on $S$.
 Let $\Sigma_0$ denote the subsurface of $S$ filled by $\cup_i \sigma_i$, i.e.\ adjoin to  $\cup_i \sigma_i$ all simply connected complementary regions. 
 Then essential annuli in $M_r$ are of three kinds after homotopy:
 \begin{enumerate}
 \item Essential annuli of the form $\sigma \times I$, respecting the product structure of $S \times I$, where
 $\sigma$ is an essential, possibly immersed curve in $S$. In this case, $\sigma$   necessarily lies in (a component of) $S \setminus \Sigma$.
 \item Essential annuli with both boundary curves on either $S \times \{0\}$ or 
  $S \times \{1\}$ with core curve homotopic to a multiple of one of the $\sigma_i$'s. In this case, $A$ wraps around $N_\ep(\sigma_i)$ finitely many times. Such an annulus is referred to as an \emph{atomic wrapping back-tracking annulus}.
  \item Essential annuli with one boundary curve on  $S \times \{0\}$, and one on
  $S \times \{1\}$ with core curve homotopic to  a multiple of  one of the $\sigma_i$'s, such that  $A$ wraps around $N_\ep(\sigma_i)$ finitely many times.
  Such an annulus is referred to as an \emph{atomic wrapping, non-back-tracking annulus}.
 \end{enumerate}
 
 Next, suppose that $M$ fibers over the circle with monodromy $\Phi$. Let $S_x \subset M$ denote the unique singular fiber. Let $S_1, \cdots, S_m$ denote drilled
 fibers and $\sigma_i \subset S_i$, $i=1, \cdots, m$ denote drilled simple closed curves. Let $M_r$ denote $M$ after drilling. Let $M_\Z$ denote the cover of $M$ corresponding to $\pi_1(S)$ and $M_S$ denote the elevation of $M_r$ to $M_\Z$.
 Let $\cdots, S_x^{-1}, S_x^0, S_x^1, \cdots$ denote the elevations of $S_x$ to $M_S$.
 Without loss of generality, assume that the elevated annulus $A_1$ starts on $S_x^0$.
 Also, let $\Sigma_n$ denote the subsurface of $S$ filled by $\bigcup_{i=1, \cdots, m}
 \bigcup_{j=0,\cdots, n-1} \Phi^j( \sigma_i)$, i.e.\ adjoin to  $\bigcup_{i=1, \cdots, m}
 \bigcup_{j=0,\cdots, n-1} \Phi^j( \sigma_i)$ all simply connected complementary regions of $S$. 
 
 \begin{lemma}\label{lem-fill}
 There exists $N \in \natls$ such that $\Sigma_n=S$ for all $n \geq N$.
 \end{lemma}
 
 \begin{proof}
 Since $\Phi$ is a pseudo-anosov, there exists $N \in \natls$ such that $$d_{\CC(S)}
 (\sigma_i, \Phi^n(\sigma_i))\geq 3$$ for all $n \geq N,$ where $d_{\CC(S)}$ denotes distance in the curve complex of $S$. Hence, $\sigma_i, \Phi^n(\sigma_i)$ fill $S$ for all for all $n \geq N$.
 \end{proof}

 Essential annuli in $M_r$ for $M$ a fibered 3-manifold will be described in greater detail and in a more general setting in Section~\ref{sec-essentialannuli} below. For now, it suffices to say that  an essential annulus in $M_r$ lifts to an essential annulus
 in $M_S$ starting and ending on $S_x^i, S_x^j$ for some $i, j \in \Z$.
 \begin{comment}
 	We would like to consider nearest-point projections of distinct elevations of singular fibers on each other. The following is a consequence of the annulus
 	theorem.

 	\label{lem-nppbdd}
There exists $C\geq 0$ such that the following holds. Let $\til S_1', \til S_2'$ be two distinct lifts of an undrilled fiber $S'$. Then 
 the nearest point projection of $\til S_1'$ on $ \til S_2'$ has diameter bounded above by $C$. Similarly, $M$ is $S \times I$, then the nearest point projection
 of any elevation of $S \times \{0\}$ or $S \times \{1\}$ on any other also 
 has diameter bounded above by $C$. 
 \end{comment}

\subsection{Essential annuli in $F$}\label{sec-essentialannuli} We use the classification of essential annuli in $M_r$ described above to identify essential 
 annuli in the drilled bundle $F$. Since $F \subset E$, an essential annulus $A \subset F$ may also be regarded as a subset of $E$. The formalism we use is due to Bestvina-Feighn \cite{BF}.

 \begin{defn}\label{def-hallway}\cite[p. 87]{BF} 
 	Let $\GG$ be a graph, and $Y$
 	a graph of spaces with base graph $\GG$ so that the maps of edge-spaces to vertex spaces are injective at the level of the fundamental group. Let $X=\til Y$ be the universal cover, and
 	$ \TT$ the resulting tree of spaces, whose vertex and edge   spaces are universal covers of vertex and edge   spaces of $Y$.
 	Let \(m\) be a positive integer and $I$ denote the closed unit interval. A \emph{hallway} in \({X}\) is a map \(\mathcal{H}:[-m, m] \times I \rightarrow {X}\) with the following properties.
 	\begin{enumerate}
 		\item \(\mathcal{H}^{-1}(\cup {X_e}) = \{-m, -m + 1, \dots, m\} \times I\)
 		\item \(\mathcal{H}\) is transverse  to \(\cup {X_e}\) relative to the previous condition.
 		\item For each \(i \in \{-m, -m + 1, \dots, m\}\), the image of \(\mathcal{H}\) restricted to \(\{i\} \times I \) is a geodesic in the corresponding edge space.
 	\end{enumerate}
 	We say that such a hallway  has \emph{length} \(2m\). The \emph{girth} of the hallway is the length of the curve \(\mathcal{H}(\{0\} \times I)\).
 		A hallway is \emph{essential} if the projection of \(\mathcal{H}\) onto the base tree \({\TT}\) is a path which does not backtrack.
 	
 		A map \({A}: [-m, m] \times S^1 \rightarrow Y\) is an \emph{annulus} if the lift \(\til{{A}}: [-m, m] \times I \rightarrow {X}\) is a hallway. An annulus $A$ is \emph{essential} if the hallway $\til{{A}}$ is essential.
 		
 		An essential annulus  \({A}: [0, m] \times S^1 \rightarrow Y\) is said to \emph{start at $Y_v$} if  ${A}\big(\{0\} \times S^1\big) \subset Y_v$.
 \end{defn}
To relate the notion of essential annuli in 3-manifolds with those in Definition~\ref{def-hallway}, we need the following.
 \begin{defn}\label{def-atom}
 	Recall that $\Pi: E \to \GG$ denotes the surface bundle $E$ over the graph $\GG$.
 	Let $M_e = \Pi^{-1} (\bar{e}) \cap F$. If $e$ is a drilled edge, we refer to 
 	$M_e$ as a \emph{drilled atom} of $F$, else we refer to it as an \emph{undrilled atom} of $F$. Elevations of drilled (resp.\ undrilled) atoms of $F$ to $\til E_r$ will be referred to as drilled (resp.\ undrilled) atoms of $\til E_r$.
 	Elevations of drilled (resp.\ undrilled) atoms of $F$ to $\til F$ will be referred to as drilled (resp.\ undrilled) atoms of $\til F$.\\
 	For any subgraph $\LL \subset \GG$, $M_\LL = \Pi^{-1} (\LL) \cap F$ will be called the restriction of $F$ to $\LL$. If $\LL$ is a maximal subgraph of $\GG$, such that each edge of $\LL$ is undrilled, $M_\LL $ will be called a \emph{maximal undrilled subbundle} of $F$ (as well as of $E$). Elevations of such $M_\LL $
 	to $\til E_r$ or $\til F$ will be called \emph{elevated maximal undrilled subbundles}.
 \end{defn}
Recall from Section~\ref{sec-prel-bdl} that, strictly speaking, $F$ admits a graph of spaces structure, with underlying graph $\GG^*$, where
  the vertex spaces of $F$ are drilled or undrilled atoms. However, for ease of notation, we shall henceforth refer to $\GG^*$ by $\GG$, abusing notation slightly.
   Let $F_S$ be the cover of
  $F$ corresponding to the kernel of the map $\Pi_*: \pi_1(F) \to \pi_1(\GG)$. 
  Let $\til \GG$  denote the universal cover of $\GG$.
Let $\Pi_S: F_S \to \til \GG$ denote the projection map from  $F_S$ to $\til \GG$.
Then $F_S$ admits a graph of spaces structure, where 
\begin{enumerate}
\item   each vertex space is a copy of $S \times I$ with, possibly, some curves drilled. The interior of each vertex space is given by one of the following:
\begin{itemize}
\item $M \setminus S_x$ if $M$ is an undrilled atom of $F$ fibering over the circle,
\item $M_r \setminus S_x$ if $M_r$ is a drilled atom of $F$ fibering over the circle,
\item $M \setminus S\times \{0,1\}$ if $M$ is an undrilled atom of $F$ fibering over $I$,
\item $M_r \setminus S\times \{0,1\}$ if $M_r$ is a drilled atom of $F$ fibering over $I$.
\end{itemize}
\item each edge space is $S$.
\end{enumerate}
 Thus, $F_S$ is a drilled surface bundle over the graph $\til \GG$ (see Definition~\ref{def-drillbdl}). 
 Note that $E_S$ also has a similar graph of spaces decomposition over $\til \GG$ where all the vertex spaces are  copies of $S\times I$, and each edge space is $S$. With this structure in mind, $F_S$ embeds into $E_S$ as a graph of spaces over $\til \GG$.
 We now describe essential annuli in $F_S$ in terms of
 the description of  essential annuli in atoms given in Section~\ref{sec-drill3}.

\begin{comment}
Recall from Section~\ref{sec-drill3} that the undrilled atoms $M_e$ of $F$
are of two kinds:
\begin{enumerate}
	\item $S \times I$
	\item a hyperbolic 3-manifold $M$ fibering over the circle with fiber $S$ 
\end{enumerate}
\end{comment} 
 
Let $A$ denote an essential annulus in $F$. Let $A_S$ denote an elevation of $A$ to $F_S$. Then the image $\Pi_S(A_S)$ is an unparametrized geodesic $[a,b]$ in $\til \GG$. Let $a=v_0, v_1, \cdots, v_n =b$ denote the vertices of $[a,b]$. Let $F_S^i$ denote the vertex space of $F_S$ corresponding to $v_i$. 
Then  $A_S$  is given by a concatenation of essential annuli in the atoms $F_S^i$. We shall show now that  $A_S$ comes broadly in  two flavors: backtracking
and non-backtracking.
Suppose $A_S$ is  an essential annulus  of the form $\sigma \times [p, q]$, respecting the topological product structure of $F_S$, such that
\begin{enumerate}
\item $\sigma$  an essential, possibly immersed curve in $S$,
\item  $\sigma \times \{0\}$
lies on $S_x^0$, and $\sigma \times \{n\}$
lies on $S_x^n$,
\end{enumerate}
then we shall refer to such an essential annulus as a \emph{non-backtracking annulus}.
 	In this case, $\sigma$   necessarily lies in (a component of) $S \setminus \Sigma_n$ (for instance, by the annulus theorem). In this case, there are no atomic wrapping annuli in the sense of Section~\ref{sec-drill3} contained in $A_S$.

Next, suppose $A_S$ is  an essential annulus  containing an atomic wrapping annulus.
Recall that $E_S$ denotes the surface bundle over $\til\GG$ corresponding to the cover of $E$ corresponding to $\pi_1(S)$, so that $F_S$ is obtained from $E_S$ by drilling a family of non-homotopic simple closed curves.
Then, there exists a drilled curve $\eta$ in $E_S$ such that $A_S$ wraps around
the boundary of $N_\ep (\eta)$ in $F_S$. Hence, the core curve of $A_S$ is homotopic to a non-trivial power of $\eta$. We observe the following:

\begin{lemma}\label{lem-1wrap}
Suppose $A_S$ is  an essential annulus  containing an atomic wrapping annulus.
Then  $A_S$ contains exactly one atomic wrapping annulus.
\end{lemma}

\begin{proof}
 Since $A_S$ is  an essential annulus  containing an atomic wrapping annulus, it has core curve freely homotopic to a (non-trivial power of a) drilled curve $\eta$. Note next that, by Definition~\ref{def-drillbdl},
 and Condition~\ref{cond-hyp}, no  two distinct elevations of drilled curves to $E_S$ are freely homotopic. Hence,  $A_S$ contains exactly one atomic wrapping annulus.
\end{proof}

Essential annuli $A_S$ in $E_S$ containing an atomic wrapping annulus are of two kinds depending on the nature of the atomic wrapping annulus. Let $B_S\subset A_S$ denote the
unique atomic wrapping annulus contained in $A_S$ with core curve homotopic to
$\eta^m$, $m \neq 0$.
\begin{enumerate}
	\item Both boundary curves of $B_S$ lie on  a single singular fiber  $S_y$. In this case,
	$B_S$, and hence $A_S$, wraps around $N_\ep( \eta)$ finitely many times. We shall refer to such a $B_S$  as an \emph{ elementary backtracking annuli}, and $A_S$ as an annulus with backtracking.
	\item There exists a drilled atom $M_e$ with boundary surfaces $S_y, S_z$, such that the distinct boundary curves of $B_S$ lie on the distinct 
	boundary surfaces $S_y, S_z$.  Again,  	$B_S$, and hence $A_S$ wraps around $N_\ep( \eta)$ finitely many times, but otherwise is required to respect the product structure on 
$M_e$.  We shall refer to such $B_S$ as \emph{ elementary wrapping annuli without backtracking}, and $A_S$ as an annulus with wrapping but no backtracking.
\end{enumerate} 

\begin{defn}\label{def-intersect}
An essential annulus  $A_S$ in $F_S$ is said to \emph{intersect} a drilled atom $M_r \subset F_S$ if $A_S$ contains an elementary annulus $B_S \subset M_r$. (Note that $B_S$ may be a non-backtracking annulus, an elementary backtracking annulus, or an  elementary wrapping annulus without backtracking.)
\end{defn}

\begin{comment}\label{def-breadth}
	We fix a homeomorphism $E_S \to S \times \til{\GG}$ that is equivariant with respect to the action of $\pi_1(\GG)$. 
	Let $\Pi: E_S \to \GG$ denote the induced projection from the subset $E_S$ to $\til{\GG}$. 
	Let $A_S \subset F_S \subset E_S$ be an essential 
	annulus. Then there exist vertices $a, b \in \til{\GG}$ such that
	\begin{enumerate}
		\item $A \subset F_S \cap (S \times [a,b]$, where $[a,b] \subset \til{\GG}$ denotes the geodesic between $a, b$ in $\til{\GG}$.
		\item  $ [a,b]$ is the minimal interval with integral end-points satisfying the first condition above.
	\end{enumerate}
	We refer to $d_{\til{\GG}} (a,b)$ as the \emph{breadth} of $A_S$.
\end{comment}

\begin{lemma}\label{lem-bddbreadth}
		There exists $D \in \natls$ such that the following holds:\\
	Let $A_S \subset F_S \subset E_S$ be an essential 
	annulus intersecting drilled atoms $M_a$, $M_b$, where $a, b \in \til{\GG}$. Then
$$d_{\til{\GG}} (a,b) \leq D.$$
\end{lemma}

\begin{proof}
	This is similar to Lemma~\ref{lem-fill}. It was shown by Kent-Leininger \cite{kl-coco} and Hamenstadt \cite{hamen} that the following are equivalent:
	\begin{enumerate}
	\item $\pi_1(E)$ is hyperbolic,
	\item  $\pi_1(\GG)$ acts on the curve complex $\CC(S)$ with qi-embedded orbits,
	\item $\pi_1(\GG)$ is a convex-cocompact subgroup of the mapping class group $MCG(S)$ in the sense of \cite{farb-coco}.
	\end{enumerate}
	Since, $\pi_1(E)$ is hyperbolic by assumption, it follows that  
	 $\pi_1(\GG)$ acts on the curve complex $\CC(S)$ with qi-embedded orbits. Let $i: 
	 \til{\GG} \to \CC(S)$ be the induced qi-embedding.
	 Hence,
	 there exists $D \in \natls$  such that if $d_{\til{\GG}} (a,b) > D,$ then 
	 $d_{\CC(S)} (i(a),i(b)) > 3$. 
	 Let $\sigma_a$ and  $\sigma_b$ denote drilled curves in $M_a$, $M_b$ respectively,
	 so that $\sigma_a$ and  $\sigma_b$ fill $S$. 
	 
	 If the core curve of $A_S$ 
	intersects both $M_a$, $M_b$  in the sense of Definition~\ref{def-intersect}, it must be disjoint from the subsurface of $S$ filled by  $\sigma_a$ and  $\sigma_b$,
	an impossibility.
\end{proof}

\subsection{Partial electrification}\label{sec-pel} Recall that the bundle $\Pi: E \to \GG$   restricts to a map
$\Pi_r: F \to \GG$. \\

\noindent {\bf Partial electrification of $  \partial{N_\ep(\sigma_i)}$:}\\
Let $M_r$ denote a drilled atom in $F$, and let $ \partial  {N_\ep(\sigma_i)}$ denote the boundary of a drilled curve in $M_r$.  Choose a homeomorphism of $ \partial  {N_\ep(\sigma_i)}$ with $S^1 \times S^1$, and assume, without loss of generality, that each $S^1 \times \{t\}$ is a meridian. Equip $S^1 \times S^1$ with a product metric,
where the first factor is given the zero metric, and the second factor the standard round metric of radius one. We refer to the resulting path-pseudometric on 
$ \partial  {N_\ep(\sigma_i)}$ as \emph{the partially electrified path-pseudometric}
and denote it as $\dpe$.\\

\begin{comment}
 Lift this product structure equipped  to
$\partial \til{N_\ep(\sigma_i)}\cong S^1 \times \R$.
Now, equip each   $\partial \til{N_\ep(\sigma_i)}$ with a pseudometric $\dpe$ given by the product of the zero metric on the $S^1$ factor and the Euclidean metric on the $\R-$factor. We refer to $\dpe$ as the \emph{partially electrified metric} on
$\partial \til{N_\ep(\sigma_i)}$. \\
\end{comment}

 \noindent {\bf Partial electrification of drilled atoms:}\\ Recall that any drilled
 $M_r$ of $F$ has been equipped with the restriction of a complete hyperbolic metric.
 If $M_r$ is obtained by drilling $S \times I$, then the surface boundary components $S \times \{0,1\}$ are convex. Removing a small neighborhood of the cusps we obtain boundary tori 
 $ \partial  {N_\ep(\sigma_i)}$  equipped
 with  flat Euclidean metrics. Abusing notation slightly, we continue to refer to the
 resulting compact 3-manifold with boundary also as $M_r$.
 Rescaling the hyperbolic metric if necessary, we may assume that $ \partial  {N_\ep(\sigma_i)}$ has a product metric as in the previous paragraph. Replace each such metric by the partially electrified path-pseudometric $\dpe$ described in the previous paragraph. We now consider a family of paths, each of which is given by a concatenation of pieces that 
 \begin{enumerate}
 \item Either have interior disjoint from the boundary tori $\{ \partial  {N_\ep(\sigma_i)}\}$. The length of such a piece is given by its hyperbolic length. 
 \item Or lie entirely in some boundary torus $\{ \partial  {N_\ep(\sigma_i)}\}$.
 The length of such a path is given by its length with respect to $\dpe$ on 
  $\{ \partial  {N_\ep(\sigma_i)}\}$.
 \end{enumerate}
 Then the length of a path is given by the sum of the above pieces. The resulting 
 path-pseudometric on $M_r$ is referred to as  \emph{the partially electrified path-pseudometric on $M_r$ }
 and is also denoted  as $\dpe$. 
 
 If the atom  $M_r$ of $F$ is obtained from an atom $M$ of $E$ that fibers over the circle, then any elevation $M_S$ of $M_r$ to $F_S$ is a cyclic cover of $M_r$ corresponding to the natural epimorphism to $\Z$ inherited from $M$. Then 
  $M_S$ is a concatenation of a $\Z$'s worth of atoms of $F_S$. Each atom $M_e$ of $F_S$ is equipped with the inherited path metric from $M_S$, and the resulting 
  path-pseudometric on $M_e$ is referred to as  \emph{the partially electrified path-pseudometric on $M_e$ }
  and is also denoted  as $\dpe$. 
 
 The universal cover of $M_r$ will be denoted by $\til M_r$. The lift of the partially electrified path-pseudometric on $M_r$ to $\til M_r$ is referred to as  \emph{the partially electrified path-pseudometric on $\til M_r$ }
 and is also denoted  as $\dpe$. Similarly for $M_e$. 
 
 \begin{lemma}\label{lem-atomshyp} There exist $\delta \geq 0, K \geq 1, \ep \geq 0$ such that the following hold.\\
 	Let $M_e$ denote an atom of $F_S$. Let $S_x$ denote  a surface boundary component
 	of $M_e$. Let $\PP$ denote the collection of elevations of the tori $\{ \partial  {N_\ep(\sigma_i)}\}$ to $\til M_e$. Then
 	\begin{enumerate}
 		\item $\til M_e$ is strongly hyperbolic relative to $\PP$.
 	\item  $(\til M_e,\dpe)$ is $\delta-$hyperbolic.
 	\item the inclusion of $\til S_x$ into $(\til M_e,\dpe)$ is a $(K,\ep)-$qi embedding
 	for any elevation  $\til S_x$  of $S_x$.
 	\end{enumerate}
 \end{lemma}
 
 \begin{proof} Since a hyperbolic structure on $M_e$ may be chosen so that it has
 	convex boundary,
 	$\til M_e$ is strongly hyperbolic relative to $\PP$. (This is a consequence of the main theorem of \cite{farb-relhyp}.)
 	 The second conclusion is then a special case of \cite[Lemma 1.20]{mahan-pal}.

 	 Let $d$ denote the metric on $\til M_e$ lifted from the intrinsic path-metric on $M_e$. Let $d_e$ denote the electric metric on $\til M_e$ after electrifying
 	 the elements of $\PP$ as in \cite{farb-relhyp}.
 For $u, v \in \til S_x$, let $\gamma_{uv}, \gamma_{uv}^e, \gamma_{uv}^p$ denote geodesics with respect to $d, d_e,\dpe$ respectively.
 The second conclusion will follow from the following two facts:
 \begin{enumerate}
 \item $\gamma_{uv}, \gamma_{uv}^e, \gamma_{uv}^p$  track each other (uniformly, independent of $u, v$) away from $\PP$ by \cite[Lemma 1.21]{mahan-pal},
 and Lemma 4.5 and Proposition 4.6 of \cite{farb-relhyp}. (See Lemma~\ref{lem-track} below for a slightly more general statement.)
 \item the nearest-point projections of elements of $\PP$ equipped with $\dpe$ onto
 $\til S_x$ are either uniformly bounded in diameter, or
  uniform quasi-isometric embeddings. 
 \end{enumerate}
 In fact, any $P \in \PP$ is stabilized by a conjugate of 
 $\pi_1(N_\ep(\sigma_i))=\Z^2$ for some $i$. Let $stab(P)$ denote the stabilizer of $P$.
 Then, $stab(P) \cap \til S_x$ is either trivial or infinite cyclic.
 
 In the former case,   $\gamma_{uv}^e$ (after a small isotopy if necessary) may be assumed to be disjoint from $P$. In the latter case, if $\gamma_{uv}^e$ does enter
and exit $P$ at $y, z$ respectively, then there exist $y_1, z_1 \in  \til S_x$ such that
the geodesic joining $y_1, z_1 $ lies uniformly close to an elevation of
$\sigma_i$. It follows that $\gamma_{uv},  \gamma_{uv}^p$ track each other throughout
their lengths. The third conclusion  follows.
 \end{proof}

  \noindent {\bf Partial electrification of $F, \til F$:}\\
  The metric on each of the atoms after drilling (and before partial electrification) is denoted by $d$. Equip $F$ with a $C^0$ piecewise Riemannian metric that is bi-Lipschitz to the metric $d$ on each of the atoms. We refer to this metric on $F$ also by $d$. Now, consider rectifiable paths $\sigma: [0,1] \to F$ consisting of finitely many pieces, each of which is  contained
  in an atom.  The length of $\sigma$ 
  is declared to be the sum of the lengths of these subpaths. Replacing $d$ on each atom by the partially electrified metric $\dpe$ on that atom, we obtain a 
  partially electrified path pseudo-metric, also denoted  $\dpe$, on $F$.
  
  Lifting $d, \dpe$ to the universal cover $\til F$ we obtain a metric $d$ and a 
   partially electrified path pseudo-metric  $\dpe$ respectively. The distance between
   a pair of point $u, v$ is then obtained by passing to an infimum over all  paths $\sigma$ as above joining $u, v$.
   
   \begin{rmk}\label{rmk-pel}
   	The partial electrification construction above is adapted from \cite{mahan-reeves} (see \cite[Section 1.3]{mahan-pal} for a summary).
   \end{rmk}

\subsection{Partial electrification and relative hyperbolicity}\label{sec-mainrh}

Let $\til F$ denote the universal cover of $F$. We lift the pseudo-metric $\dpe$ to 
$\til F$ to obtain a partially electrified metric on $\til F$ denoted again by 
$\dpe$. The following is the main theorem of this section.

\begin{theorem}\label{thm-tilfdpelhyp}
	$(\til F, \dpe)$ is hyperbolic.
\end{theorem}

To prove Theorem~\ref{thm-tilfdpelhyp}, we shall use the Bestvina-Feighn combination Theorem~\ref{bestvina-feighn}.

\subsubsection{The Bestvina-Feighn Combination Theorem}\label{sec-bf}

 \begin{defn}\label{def-bf}\cite{BF} Let $\GG$ be a graph, and $Y$
	a graph of spaces with base graph $\GG$ so that the maps of edge-spaces to vertex spaces are injective at the level of the fundamental group. Let $X=\til Y$ be the universal cover, and
	$ \TT$ the resulting tree of spaces, whose vertex and edge   spaces are universal covers of vertex and edge   spaces of $Y$. Suppose that the following hold: 
	\begin{enumerate}
		\item vertex spaces \(\{X_v\}\) and edge spaces \(\{X_e\}\) are all $\delta-$hyperbolic metric spaces for some $\delta>0$
		\item there exist $K \geq 1, \ep \geq 0$ such that the the maps of edge-spaces to vertex spaces for $X$ are all $(K, \ep)-$quasi-isometric embeddings.
	\end{enumerate} 
	Then $Y$ is said to be a graph of hyperbolic metric spaces satisfying the  qi-embedded condition.
\end{defn}

Recall the notion of  hallways and annuli from Definition~\ref{def-hallway}.

\begin{defn}\label{def-hyphallway}\cite[p. 87]{BF} 
	For \(\lambda > 1\), a hallway  \(\mathcal{H}\) is said  to be \emph{\(\lambda\)-hyperbolic} if
\[
l(\mathcal{H}(\{0\} \times I)) \leq \frac{1}{\lambda}\, \text{max}\{l(\mathcal{H}(\{-m\} \times I)), l(\mathcal{H}(\{m\} \times I))\},
\]
where \(l\) denotes the length of the path in the corresponding edge space.

	Let \(\rho > 0\). Given \(i \in \{-m, -m + 1, \dots, m\}\), denote the vertex space that \(\mathcal{H}([i, i+1] \times I)\) lies in by \({X_{v_i}}\). The hallway is \emph{\(\rho\)-thin} if for all such \(i\) and for any \(t \in I\), \(d_{{X_{v_i}}}(\mathcal{H}(i, t), \mathcal{H}(i+1, t)) < \rho\).
	
	 The \emph{girth} (resp.  \emph{length})  of the annulus \(\mathcal{A}\) is the girth  (resp.  \emph{length})  of the induced hallway \(\widetilde{\mathcal{A}}\). 

\end{defn}
	Similarly,  the rest of the terminology, i.e.\  hyperbolicity, thinness, essentiality, for 
	the annulus \(\mathcal{A}\), is defined via \(\widetilde{\mathcal{A}}\).\\
	
\noindent {\bf The annuli flare condition:}
The graph of spaces $Y$ (with base graph $\GG$)  satisfies the \emph{annuli flare condition} if there exists \(\lambda > 1\)  and \(m \geq 1\) such that the following holds: given any \(\rho > 0\), there exists a constant \(H(\rho)\) so that whenever \(\mathcal{A}\) is a \(\rho\)-thin essential annulus of length \(2m\) and girth at least \(H(\rho)\), \(\mathcal{A}\) is \(\lambda\)-hyperbolic.  The graph of spaces $Y$ satisfies the  \emph{weak annuli
flare condition} if there are numbers $\lambda > 1, m > 1$, and $H$ such that
any $4\delta$-thin essential annulus of length $2m$ and girth at least $H$ is
$\lambda-$hyperbolic. The following statement gives  the Bestvina-Feighn Combination Theorem in a consolidated form.

\begin{theorem}\cite{BF}\cite[Theorem 3.2]{BFcorr}\label{bestvina-feighn}
	Let $\Pi: Y \to  \GG$ be a graph of spaces whose vertex and edge spaces have $\delta-$hyperbolic universal covers for some $\delta > 0$. If $\Pi: Y \to  \GG$ satisfies the following conditions,  then the universal cover $X$ of $Y$ is
	hyperbolic:
	\begin{enumerate}
	\item the quasi-isometrically embedded condition 
	\item the annuli flare
	condition or the weak annuli flare
	condition.
	\end{enumerate}
\end{theorem}

The following definition adapts \cite[Definition 4.26]{mahan-tight} to our context.
\begin{defn}\label{def-flareinall}
	We say that,  a vertex space $Y_v$  \emph{flares in all directions weakly} 
	if there are numbers $\lambda > 1, m > 1$, and $H$ such that if $\AAA: S^1 \times [0,m] \to Y$ is any $4\delta$-thin essential annulus of length $m$ and girth at least $H$ \emph{starting at $Y_v$} (in the sense of Definition~\ref{def-hallway}), then 
	the associated lifted hallway  \(\mathcal{H}\) satisfies
	\[
	l(\mathcal{H}(\{0\} \times I)) \leq \frac{1}{\lambda}\, l(\mathcal{H}(\{m\} \times I)).
	\]
\end{defn}

We shall need a modification of Theorem~\ref{bestvina-feighn} to guarantee global quasiconvexity of a vertex space.
The following now adapts \cite[Proposition 4.27]{mahan-tight} to our context. 

\begin{cor}\label{cor-qc}
Let $\Pi: Y \to  \GG$ be a graph of spaces satisfying the conditions of Theorem~\ref{bestvina-feighn}. Further, let $Y_v \subset Y$ be a vertex space that flares in all directions. Then $\til Y_v$ is quasiconvex in $\til Y$.
\end{cor}

\begin{proof}
Since the context of \cite[Proposition 4.27]{mahan-tight} is slightly different, we sketch the mild modifications necessary. We recall a construction from \cite[Section 3]{mitra-trees}
here. It follows from \cite[Theorem 3.8]{mitra-trees} that there exists $C >0$ such that the following holds.
Given any geodesic $\lambda \subset (\til Y_v, d_v)$, there exists a
`ladder' $\LL_\lambda \subset \til Y$ containing $\lambda$, such that 
$\LL_\lambda$ is $C-$quasiconvex. Hyperbolicity of $(\til Y, d)$ now guarantees that for all such geodesics $\lambda \subset (\til Y_v, d_v)$,  $\LL_\lambda$ is hyperbolic.
The construction of $\LL_\lambda$ in \cite[Section 3]{mitra-trees} now shows that $\til \Pi: \til Y \to  \til \GG$ restricts to $\til \Pi_\lambda: \LL_\lambda \to  \TT$, where $\TT\subset \til \GG$
is a tree. Further, for any vertex $w\in \TT$, $\til \Pi_\lambda^{-1}(w)$ is a geodesic segment in the vertex space $\til \Pi^{-1}(w) \subset \til Y$. Thus, $\LL_\lambda$ is a tree
of spaces, where each vertex space is isometric to an interval. The hypothesis that  $Y_v \subset Y$ be a vertex space that flares in all directions guarantees that $\LL_\lambda$ flares
in all directions also. Hence $\lambda$ is a quasigeodesic (with uniform constants) in $\LL_\lambda$. Since $\LL_\lambda$ is $C-$quasiconvex, 
 $\lambda$ is a quasigeodesic (with uniform constants) in $\til Y$.
\end{proof}

\begin{rmk}\label{rmk-qc}
More generally, if $Z \subset Y_v$ is a subspace such that 
\begin{enumerate}
	\item the inclusion induces an injection at the level of fundamental groups,
	\item $\til Z$ is qi-embedded in $\til Y_v$,
\end{enumerate}
then an auxiliary vertex $w$ and an edge $e=[w,v]$ may be added to the base graph $\GG$, so that $Y_w=Y_e=Z$. Definition~\ref{def-flareinall} and Corollary~\ref{cor-qc} may thus be applied to such subspaces $Z$ of $Y_v$ as well.
If $Z$ flares in all directions, then $\til Z$ is quasiconvex in $\til Y$
by Corollary~\ref{cor-qc}.
\end{rmk}

We also record the following, where we assume implicitly that there is a cocompact group action so that  the annuli flare
condition makes sense. (We also refer the reader to \cite[p. 90]{BF} and \cite[Section 4]{BFcorr} for an analogous hallways flare
condition.)

\begin{lemma}\label{lem-subtree}
Let $\Pi_\TT: X \to \TT$ be a tree of spaces obtained as a universal cover of a 
graph of compact spaces. Suppose that each vertex space $X_v$ and edge space $X_e$ 
of $X$ is $\delta-$hyperbolic. Further, suppose that  the following conditions are satisfied:
\begin{enumerate}
\item the quasi-isometrically embedded condition 
\item the annuli flare
condition or the weak annuli flare
condition.
\end{enumerate}
Let $\TT_0$ be a subtree of $\TT$ and $X_0 = \Pi_{\TT}^{-1}(\TT_0)$. Then $X_0$ is hyperbolic.
\end{lemma}

\begin{proof}
We note that any hallway in $X_0$ is also a hallway in $X$. In particular, 
$\Pi_\TT: X_0 \to \TT_0$ is a tree of spaces satisfying 
\begin{enumerate}
	\item the quasi-isometrically embedded condition 
	\item the annuli flare
	condition or the weak annuli flare
	condition.
\end{enumerate}
The Corollary is now a direct consequence of Theorem~\ref{bestvina-feighn}.
\end{proof}

\subsubsection{Hyperbolicity of partially electrified bundle: Proof of Theorem~\ref{thm-tilfdpelhyp}}\label{sec-hypofdpe} $ $\\
To prove Theorem~\ref{thm-tilfdpelhyp}, it suffices 
to check the two conditions of Theorem~\ref{bestvina-feighn}.\\

\noindent {\bf Hyperbolicity of vertex spaces and the quasi-isometrically embedded condition:} This follows from Lemma~\ref{lem-atomshyp}. 
\\

\noindent {\bf Identifying  $\rho-$thin annuli:} It remains  to prove the annuli flare condition.  We recall the description of essential annuli in $F_S$ from Section~\ref{sec-essentialannuli}. Let $D$ be as in Lemma~\ref{lem-bddbreadth}. We choose $m$ in the annuli flare condition so that $2m > D$. Hence, any essential annulus
$A_S$ in $F_S$ can intersect \emph{at most one drilled atom}. Thus, any essential annulus
$A_S$ in $F_S$ of length $2m$ can be of exactly one of the following three mutually
exclusive types:
\begin{enumerate}
\item[Case 1:] $A_S$ is a non-backtracking annulus with core curve having free homotopy type distinct from any of the drilled curves,
\item[Case 2:] $A_S$ contains an elementary wrapping annulus $B_S$ without back-tracking.
Here, $B_S$ wraps  around
$\partial N_\ep(\eta)$ for some $\eta$. The core curve of $A_S$ is  then freely homotopic to a (non-trivial power of)  $\eta$.
\item[Case 3:] $A_S$ contains an elementary back-tracking annulus $B_S$ wrapping around
$\partial N_\ep(\eta)$ for some $\eta$. The core curve of $A_S$ is  then freely homotopic to a (non-trivial power of)  $\eta$. 
\end{enumerate}

In cases 2, 3 above $A_S$ is the concatenation of 3 pieces:
\begin{enumerate}
\item the  elementary wrapping annulus $B_S$ (with or without back-tracking),
\item two non-backtracking annuli $A_S^\pm$, such that $A_S^\pm \cap B_S$ consist of 
 curves $\eta^\pm$ that are freely homotopic in $F_S$. If $B_S$ is an  elementary wrapping annulus without  back-tracking, then there exist distinct singular fibers $S_x^\pm$  of $F_S$ (bounding the atom of $F_S$ containing $B_S$) such that
 $\eta^\pm \subset S_x^\pm$.
 If $B_S$ is an  elementary wrapping annulus with  back-tracking, then there exists a single singular fiber $S_x$ of $F_S$  (a boundary component of the atom of $F_S$ containing $B_S$)
 such that $\eta^\pm \subset S_x$, and $\eta^\pm$ are freely homotopic within $S_x$.
\end{enumerate}

 Since 
$B_S$ wraps around
$\partial N_\ep(\eta)$ for some drilled curve $\eta$, and since $\partial N_\ep(\eta)$ is an elevation of one of finitely many  tori in $F$, 
the core curve of $A_S$ is the same as the core curve of
$B_S$ and hence is of the form $\gamma^n$ for some $n \in \natls$, where $\gamma$ is one of the drilled curves in $E$.\\

\noindent {\bf Checking the annuli flare condition:}\\
\noindent {\it Case 1: $A_S$ is a non-backtracking annulus with core curve having free homotopy type distinct from any of the drilled curves.} \\
 We start with the following converse to the Bestvina-Feighn combination theorem.
 
 \begin{prop}\label{prop-conversebf}
$E_S$ satisfies the weak annuli-flare condition.
 \end{prop}
 
 \begin{proof}
 Since $E$ is hyperbolic, this is a special case of \cite[Proposition 5.8]{mahan-sardar}, where it is shown that $E$ satisfies a flaring condition.
 This is equivalent to hyperbolicity of hallways in $E_S$ and implies
 the weak annuli-flare condition.
 \end{proof}
 
 \begin{cor}\label{cor-weakannuliflare}
 Non-backtracking annuli  with core curve having free homotopy type distinct from any of the core curves satisfy the annuli-flare condition: more precisely, there exist $\lambda > 1$, $m>1$ such that if $A_S$ is a non-backtracking annulus with girth at least 1, it satisfies the weak annuli flare condition.
 \end{cor}
 
 \begin{proof}
Note that $A_S \subset F_S \subset E_S$. Since $A_S$ is a non-backtracking annulus in $F_S$, it is an essential annulus in $E_S$. Proposition~\ref{prop-conversebf} now gives the desired conclusion.
 \end{proof}
 
 \noindent {\it Cases 2, 3: $A_S$ is a wrapping annulus with core curve  freely homotopic to a power  of one of the drilled curves.} \\ We shall give a unified proof of these two cases.
 Let $B_S$ denote the  elementary wrapping annulus contained in $A_S$.
 Let $\gamma$ denote the drilled curve such that $B_S$ wraps around 
 $\partial N_\ep(\gamma)$. Choose an orientation of $\gamma$ and $n \in \natls$ such that the core curve  of 
 $B_S$ (and hence that of $A_S$) is freely homotopic to $\gamma^n$.
 Let $M_e$ denote the drilled atom of $F_S$ containing $B_S$.
 Let $S_x^\pm$ denote the boundary components of $M_e$, and $A_S^\pm$ denote the two components of $A_S \setminus \, Int(B_S)$. 
 Then each of the annuli $A_S^\pm$ is an essential annulus in $E_S$ with one boundary
 curve in $S_x^+ \cup S_x^-$. (If $A_S$ is back-tracking, then both boundary curves
 lie in the same surface boundary component. If $A_S$ is without back-tracking, then the boundary curves
 lie in  different surface boundary components.) Let $\theta^\pm$ denote the boundary
 curve of $A_S^\pm$ on $S_x^+ \cup S_x^-$.
 
It will help to explicate the special case where $n=1$, i.e.\ the core curve of $A_S$ is freely homotopic to $\gamma$. Then each of
 $A_S^\pm$ is a flaring annulus, with $l(\theta^\pm)$ uniformly close to the girth
 of $A_S$. This follows from the fact that the length $l(\theta^\pm)$ is uniformly bounded. Hence, in the formulation of Definition~\ref{def-hyphallway},  
 \[
 l(\mathcal{H}(\{0\} \times I)) \leq \frac{1}{\lambda}\,  l(\mathcal{H}(\{m\} \times I))\},
 \]
 for all $m$ large enough (i.e.\ we can ignore the negative direction of the hallway
 from Definition~\ref{def-hyphallway}). 
 Further, note that the $\dpe-$length of the annulus $B_S$ is uniformly bounded. This is because the meridian of $\partial N_\ep (\gamma)$ that $B_S$ wraps around has length
 zero. Hence the concatenation $A_S^+ \cup B_S \cup A_S^-$ satisfies the weak annuli flare condition.

 For general $n\in \natls$, any elevation of $\theta^\pm$ (freely homotopic to $\gamma^n$)
 to the universal cover $\til E_S$  gives a uniform quasigeodesic.  This follows from the fact that $\gamma$ is an elevation of one of the (finitely many) drilled curves. Hence, again, $l(\theta^\pm)$ is uniformly close to the girth
 of $A_S$. Again, $A_S^\pm$ satisfies the one-sided flare condition 
  \[
 l(\mathcal{H}(\{0\} \times I)) \leq \frac{1}{\lambda}\,  l(\mathcal{H}(\{m\} \times I))\},
 \]
 in the formulation of Definition~\ref{def-hyphallway}. The same argument from the previous paragraph now shows that 
  the concatenation $A_S^+ \cup B_S \cup A_S^-$ satisfies the weak annuli flare condition.\\
  
  Thus, the sufficient conditions of Theorem~\ref{bestvina-feighn} are satisfied by
  essential annuli $A_S$ in $F_S$, equipped with the partially electrified metric $\dpe$. Theorem~\ref{thm-tilfdpelhyp} follows. \hfill $\Box$
  
  \begin{rmk}
One place where the partial electrification metric $\dpe$ is essential in the above proof
is to conclude that the $\dpe-$length of $B_S$ is uniformly bounded.
  \end{rmk}

 \subsubsection{Consequences of Theorem~\ref{thm-tilfdpelhyp}}\label{sec-consqs}
  The proof of Theorem~\ref{thm-tilfdpelhyp} above gives some additional information that we shall need below. 
  Let $\til{\Pi}: \til{F_S} \to \TT$ be the tree of spaces for the universal cover $\til{F_S}$ (=$\til F$) of $F_S$. 
 Let $\TT_0 \subset \TT$ denote a subtree. Let $\til{\Pi}^{-1} (\TT_0) = X_0$, so that
 $\til{\Pi}: X_0 \to \TT_0$ is a tree of spaces.

\begin{defn}\label{def-bdydrilledcomp}
Let $v \in \TT_0$ be a  vertex such that 
\begin{enumerate}
	\item $\til M_v$ is an  atom.
	\item $\til S_x$ is a boundary component of  $\til M_v$ that is not contained in the
	boundary of any other vertex space of $\til{\Pi}: X_0 \to \TT_0$.
\end{enumerate} 
Then we say that $\til M_v$ is a \emph{boundary   atom of $X_0$} and 
$\til S_x$ is a \emph{boundary component of  $X_0$}. If, $\til M_v$ is drilled (resp.\
undrilled), we say that  $\til M_v$ is a \emph{drilled (resp.\
	undrilled) boundary   atom of $X_0$} and 
$\til S_x$ is a \emph{drilled (resp.\
	undrilled) boundary component of  $X_0$}. 
\end{defn} 

\begin{cor}\label{cor-bdydrilledqc}
Let $X_0$ be as above. Then $X_0$ is hyperbolic, and 
 any drilled boundary component $\til S_x$  of $X_0$ is $\dpe-$quasiconvex. Here
$\dpe$ denotes the restriction of the partially electrified pseudometric $\dpe$
from $X$.
\end{cor}

\begin{proof}
Hyperbolicity of $X_0$ follows from Lemma~\ref{lem-subtree} after noting that we have checked the weak annuli flare condition for $F_S$. 
Next, we attach an auxiliary vertex space $\til S_x \times I$ to $X_0$ along $\til S_x$ and an auxiliary edge $e$ to $\TT_0$ to $v$ to get a tree of spaces
$X_0\bigcup_{\til S_x} \til S_x \times I \to \TT_0 \bigcup_v e$. Here,
$\til S_x \times \{1\} \subset \til S_x \times I $ is attached to $X_0$ along
$\til S_x\subset X_0$, so that $\til S_x\subset X_0$ becomes the edge-space $X_e$
corresponding to the new edge $e$. Further, let $v_0$ denote the extra vertex introduced
in $ \TT_0 \bigcup_v e$ (corresponding to $\{0\} \in I$). Thus, the new vertex
space $\til S_x \times I$ is $X_{v_0}$.
To prove $\dpe-$quasiconvexity of $\til S_x$, it suffices, by Corollary~\ref{cor-qc}, to show that essential annuli starting on  $X_{v_0}$ (in the sense of Definition~\ref{def-hallway}) flare in all directions in the sense of Definition~\ref{def-flareinall}. But such annuli are simply essential annuli
in the bundle $E \to \GG$ (before drilling). Further, their core curves lie in $S_x \setminus \sigma_i$ for one of the finitely many drilled curves $\sigma_i$.
The conclusion now follows from Theorem~\ref{kld} which guarantees that each of these finitely many proper essential subsurfaces of $S_x$ have uniformly quasiconvex elevations, and hence that any essential annulus starting on such a subsurface flares
in all directions.
\end{proof}

Another consequence of Corollary~\ref{cor-qc} that we record is the following.

\begin{cor}\label{cor-finiteqc}
Let $\TT_1 \subset \TT$ denote a subtree of finite diameter (but not necessarily locally-finite). Let $\til{\Pi}^{-1} (\TT_1) = X_1$, so that
$\til{\Pi}: X_1 \to \TT_1$ is a tree of spaces. Then $X_1$ is hyperbolic, and 
for any singular fiber $S_x$   with  $\til S_x\subset X_1$, $\til S_x$ is $\dpe-$quasiconvex in $X_1$. 
\end{cor}

\begin{proof}
Since any essential hallway has length bounded by the diameter of $\TT_1$, hyperbolicity of $X_1$ follows from Theorem~\ref{bestvina-feighn}. The same reason guarantees $\dpe-$quasiconvexity of $\til S_x$ by Corollary~\ref{cor-qc}. 
\end{proof}

We also note the following.

\begin{lemma}\label{lem-fiberproper} Let $X_1, S_X, \til S_X$ be as above.
Then $\til S_x$ with its intrinsic metric is qi-embedded in $X_1$.
Also, $\til S_x$ properly embedded in $(\til F, \dpe)$.
\end{lemma}

\begin{proof}
	Suppose that $S_x$ is a boundary surface of an atom ${M_e}$ (drilled or undrilled) of $F_S$. 
By Lemma~\ref{lem-atomshyp}, $\til S_x$ is qi-embedded in $\til{{M_e}}$ equipped
with $\dpe$ (the case where ${M_e}$ is undrilled is obvious). Let $\til{{M_e}}$ denote the vertex space $X_v$ for $v$ a vertex of $\TT$.
Next, let $\TT_1 \subset \TT$ denote a subtree of finite diameter (but not necessarily locally-finite) containing $v$. Then Corollary~\ref{cor-finiteqc} shows that 
$\til S_x$ is $\dpe-$quasiconvex in $\til{\Pi}^{-1} (\TT_1) = X_1$.

A reprise of the proof of  Lemma~\ref{lem-atomshyp} now shows that $\til S_x$, equipped with its intrinsic metric is qi-embedded in $X_1$.
Since the finite diameter of $\TT_1$ was arbitrary, the second conclusion follows.
\end{proof}

\begin{cor}\label{cor-nobt} There exist $K\geq 1, \ep >0$ such that the following hold.
	Let $\gamma$ be a geodesic in $(\til F, \dpe)$. Let $\til S_x$ be the elevation of
	a singular fiber bounding a drilled atom $\til M_v$ of $\til F$. Suppose further, that $u, v$ are entry and exit-points on $\til S_x$ respectively of $\gamma$ into and out of $\til M_v$. Let $\gamma|[u,v]$ denote the subpath of $\gamma$ between
	$u, v$ and $\beta(u,v)$ the geodesic in $\til S_x$ (with its intrinsic metric) between $u, v$. Let
	$$\gamma'(u,v) = (\gamma \setminus \gamma|[u,v]) \cup \beta(u,v)$$ be obtained from
	$\gamma$ by replacing $\gamma|[u,v]$ by  $\beta(u,v)$. Then $\gamma'(u,v)$ is a 
	$(K,\ep)-$quasigeodesic in  $(\til F, \dpe)$.
\end{cor}

\begin{proof} Let $X_1$ denote the closure of the component of $X\setminus \til{S_x}$
	containing $\til M_v$.
By Corollary~\ref{cor-bdydrilledqc}, $\til{S_x}$ is quasiconvex in $(X_1, \dpe)$ with uniform constants (independent of $\til{S_x}$ and $X_1$). Hence, there exists $D$ such that $\gamma|[u,v]$ lies in a $D-$neighborhood of $\til{S_x}$ in  $(X_1, \dpe)$.
Hence, by Lemma~\ref{lem-fiberproper}, $\beta(u,v)$ is a $(K_1,\ep_1)-$quasigeodesic in
$(X_1, \dpe)$. 
The corollary follows.
\end{proof}

Corollary~\ref{cor-nobt} allows us to replace $\dpe-$geodesics by 
uniform $\dpe-$quasigeodesics that do not backtrack from drilled atoms in 
the following sense.

\begin{defn}\label{def-nobt}
 A $\dpe-$quasigeodesic $\gamma'$ in  $(\til F, \dpe)$ 
 is said to be a $\dpe-$quasigeodesic
 \emph{without backtracking from drilled atoms} if it satisfies the following. \\
 If $\gamma'$ enters a 
drilled atom $\til M_r$ through a boundary surface $\til S_1$, then it can
only leave  $\til M_r$ through  a boundary surface $\til S_2 \neq \til S_1$.
\end{defn}

Corollary~\ref{cor-nobt} now allows us to observe that for any $\dpe-$geodesic $\gamma$,
there exists a $(K,\ep)-$quasigeodesic $\gamma'$ without backtracking from drilled atoms
joining the end-points of $\gamma$. In fact, $\gamma'$ is obtained from $\gamma$ by
\begin{enumerate}
\item first carrying out replacements of all backtracking
segments in drilled atoms, 
\item and finally isotoping the resulting path slightly to make it disjoint from
$\til{S_x}$ as in Corollary~\ref{cor-nobt}.
\end{enumerate}

The paragraph titled `Checking the annuli flare condition' in the proof of Theorem~\ref{thm-tilfdpelhyp} gives the following further conclusion.

\begin{cor}\label{cor-maln1}
There exists $L \geq 1$ such that the following holds.\\ Let $[a,b] \subset \til{\GG}$ denote a geodesic of length at least $L$ such that $M_a, M_b$ are drilled atoms of $F_S$. Let $M_{[a,b]}$ denote the 3-manifold given by $\Pi_S^{-1}([a,b])$.
Also, let $S_a, S_b$ denote the boundary components of $M_{[a,b]}$. Let $\pi_1(S_a), \pi_1(S_b) $ denote the subgroups of $\pi_1(M_{[a,b]})$ carried
by $S_a, S_b$.
Then $\pi_1(S_a) \cap \pi_1(S_b) =\{1\}$.
\end{cor}

\begin{proof}
Since  $M_a, M_b$ are drilled atoms, $\til S_a, \til S_b$ are $\dpe-$quasiconvex
in $\til M_{[a,b]}$, equipped with the $\dpe-$metric by Corollary~\ref{cor-bdydrilledqc}. Further, the $\dpe-$quasiconvexity constant is 
uniform, independent of $a, b$. 

If $\pi_1(S_a) \cap \pi_1(S_b) \neq\{1\}$, then there exists a loop $\alpha_a \subset S_a$ freely homotopic to an 
$\alpha_b \subset S_b$. By uniform $\dpe-$quasiconvexity, any geodesic in the 
free homotopy class of $\alpha_a$ (resp.\ $\alpha_b$) must lie close to $S_a$
(resp.\ $S_b$). This forces the existence of an $L$ such that $d_{\til{\GG}} (a, b) <L$.
\end{proof}

\section{Guessing geodesics}\label{sec-gg}

\subsection{The guessing geodesics lemma}\label{sec-sisto} 
We shall need a necessary and sufficient condition for relative hyperbolicity due to Sisto \cite{sisto2012metric} building on earlier work of Bowditch \cite{bowditch-guessg}
and Hamenstadt \cite{hamen-guessg}.

\begin{defn}\label{def-cobdd} Let $(X,d_X)$ be a geodesic metric space, and $\PP$ 
	a collection of subsets.
The collection $\PP$ 
is said to be \emph{mutually bounded} if for each $K\geq 0$, there exists $B$ so that $\diam(N_K (P ) \cap
N_K (Q)) \leq B$, $\forall P\neq Q \in \PP$.
\end{defn}

In \cite{sisto2012metric}, Sisto refers to the mutual boundedness criterion above
as condition $(\alpha_1)$. We shall need the following:

\begin{theorem}\label{thm-sisto}\cite[Theorem 4.2]{sisto2012metric}\label{thm-guess}
Let $(X, d_X)$ and $ \PP$ be as above. Suppose that for all $x, y \in X$ we are given
\begin{enumerate}
\item a path $\eta (x, y)$
connecting them,
\item a closed subset $\theta (x, y) \subset \eta(x, y)$.
\end{enumerate} 
Suppose that there exists $D\geq 0$ such that 
 the following are satisfied:
 \begin{enumerate}
 \item if $d_X(x, y) \leq 2$ then $\diam(\theta(x, y)) \leq D$,
 \item Let $d_H$ denote Hausdorff distance. Then, $\forall x', y' \in \eta(x, y)$, we have
$ d_H(\theta(x', y'), \theta(x, y)\vert [x',y'] \cup \{x', y'\}) \leq D$,
 where $\theta(x, y)\vert [x',y'] = \theta (x, y) \cap \eta(x, y)\vert [x',y']$,
 \item $\forall x, y,z \in X$,  $\theta(x, y) \subset N_D(\theta(x, z) \cup \theta(z, y))$,
 \item if $x', y' \in \eta(x, y)$ do not both lie on the same
 $ P \in \PP$, then there
 exists $z \in \theta(x, y)$ between $x', y'$,
 \item The elements of $\PP$ are mutually bounded, i.e.\ for each $K\geq 0$ there exists $B\geq 1$ such that the diameter of $N_K (P_1 ) \cap
 N_K (P_2 ) \leq B$ for $P_1 \neq  P_2 \in \PP$.
 \item $\forall  k\geq 0$, there exists $K\geq 0$, such that if $d_X(x, P )\leq k$,
 and $d_X(y, P ) \leq  k$ for some $P \in \PP$; and further $d_X(x, y) \geq
 K$ then $\theta(x, y) \subset B_K (x) \cup B_K (y)$. Also, there exists 
 $z \in \theta(x, y)\cap
 N_D(P )$.
 \end{enumerate}
Then $X$ is strongly hyperbolic relative to $\PP$.
\end{theorem}

\begin{comment}
for every C there exist μ, c0, L with the following property. For any
c ≥ c0 and (1, C)−almost-geodesic β with endpoints x, y the Hausdorff
distance between trans(x, y) and transμ,c(β) is at most some L + c.
\end{comment}

Theorem~\ref{thm-sisto} says roughly that if we can guess a family of paths 
in $X$ that satisfy the conditions required of geodesics in a relatively hyperbolic space, then $X$ itself is relatively hyperbolic.

\subsection{Path families}\label{sec-ea} Recall that $\PP$ denotes the collection of elevations of $\partial N_\ep(\sigma_i)$ in $\til F$, where 
$\sigma_i$ ranges over the finitely many drilled curves in $E$, 
 Recall that
$\til F$ admits three natural pseudo-metrics in our setup:
\begin{enumerate}
\item The path-metric $d$ lifted from $F$. Recall that the metric $d$ on $F$ is the natural path-metric induced from $E$.
\item The electric metric $d_e$ obtained from $d$ by electrifying the collection $\PP$
\cite{farb-relhyp}.
\item The partially electrified metric $\dpe$ constructed in Section~\ref{sec-pel}.
\end{enumerate}

We now recall a construction from \cite[Definition 1.13]{mahan-pal}. 
 Let $\hat \lambda$ denote an electric geodesic
in  $(\til F,\de)$ joining $a, b \in \til F$. Modify $\hat \lambda$ to a path $\lambda_{ea}$ as follows. First,
$\lambda_{ea}$ coincides with  $\hat \lambda$ away from the elevations of $\partial N_\ep(\sigma_i)$ to $\til F$.
Next, for any $\til{\partial N_\ep(\sigma_i)}$ that $\hat \lambda$ intersects, let
$x_i, y_i$ denote the entry and exit points. Join $x_i, y_i$ by a geodesic in
$\til{\partial N_\ep(\sigma_i)}$ equipped with its flat Euclidean metric. The resulting path $\lambda_{ea}$ will be called an
\emph{electro-ambient quasigeodesic} in $\til F$. 

We recall the following consequence of
 \cite[Lemma 1.21]{mahan-pal},
and Lemma 4.5 and Proposition 4.6 of \cite{farb-relhyp} for easy reference.

\begin{lemma}\label{lem-track}
Let $\til M_r$ denote a drilled atom of $\til F$. Let  $a, b \in \til M_r$.
Let $\PP_r$ denote the collection of elevations of  $\til{\partial N_\ep(\sigma_i)}$ to
$\til M_r$. Then  $\til M_r$ is strongly hyperbolic relative to the collection $\PP_r$.

 Let $d, \de,\dpe$ denote the metric, electric (pseudo-)metric, and the partially electrified  (pseudo-)metric respectively on $\til M_r$.
 Let $\lambda, \hat \lambda, \lambda_{ea}, \lambda_p$ denote respectively the geodesic,
the electric geodesic, the electro-ambient quasigeodesic, and the geodesic with respect to the  $\dpe$ on $\til M_r$ joining $a, b$. 
Then
$\lambda, \hat \lambda, \lambda_{ea}, \lambda_p$ track each other away from $\PP_r$.
\end{lemma}

\begin{proof}
It follows from \cite{farb-relhyp} that $\til M_r$ is strongly hyperbolic relative to the collection $\PP_r$, since $M_r$ admits the structure of a complete hyperbolic manifold with convex boundary.

Lemma 4.5 and Proposition 4.6 of \cite{farb-relhyp} guarantee that $\lambda, \hat \lambda$ track each other away from $\PP_r$. The construction of $\lambda_{ea}$ guarantees that $\lambda_{ea}$ and $ \hat \lambda$ agree on the nose  away from $\PP_r$.  Finally,  \cite[Lemma 1.21]{mahan-pal} guarantees that $\lambda_p$ 
and $ \hat \lambda$ track each other away from $\PP_r$.
\end{proof}

\subsubsection{ Connectors in drilled atoms:}\label{sec-connatom} $ $\\
Recall that for any drilled atom $M_r$ of $F_S$, $\til M_r$ is strongly hyperbolic relative
to $\PP_r$  by Lemma~\ref{lem-atomshyp} (see Lemma~\ref{lem-atomshyp} for notation). We assume that $M_r$ has been equipped with a complete hyperbolic metric with convex boundary (as in the proof of the first conclusion of Lemma~\ref{lem-atomshyp}).
Let $\til S_1, \til S_2$ denote two boundary components of
$\til M_r$. We are interested in the nearest point projections  of $\til S_1, \til S_2$ on each other. The nearest point projection of $\til S_1$
(resp.\ $\til S_2$) onto
$\til S_2$
(resp.\ $\til S_1$) will be denoted as $\pi_{12}$ (resp.\  $\pi_{21}$).
Three possible cases arise:\\

\noindent{\bf Case 1: Product region connectors.}\\
There exists a maximal essential proper  subsurface $\Sigma$ of one of the surface boundary components of $M_r$ such that  $\Sigma\times [0,1]$, equipped with the standard product structure embeds in $M_r$, with $\Sigma\times \{0,1\} \subset \partial M_r$. We refer to $\Sigma\times [0,1]$ as a \emph{maximal product region} in $M_r$. Further,
there exists an elevation $\til{\Sigma}\times [0,1] \subset \til M_r$, such that 
$\til{\Sigma}\times \{0\} \subset \til S_1$, and $\til{\Sigma}\times \{1\} \subset \til S_2$. In this case,  $\pi_{12} (\til S_1)$ lies in a uniformly bounded neighborhood of $\til{\Sigma}\times \{1\} \subset \til S_2$, and 
 $\pi_{21} (\til S_2)$ lies in a uniformly bounded neighborhood of $\til{\Sigma}\times \{0\} \subset \til S_1$. 
 For any $z\in \Sigma$, we refer to an elevation of
 $\{z\} \times [0,1]$ to $\til M_r$ as a \emph{product region connector between 
 $\til S_1, \til S_2$}. Note that in this case,  $\til S_1, \til S_2$ are necessarily elevations of distinct boundary components of $M_r$.

 \begin{defn}\label{defn-productpair}
If  $\til S_1, \til S_2$ are connected by a product region connector,  we say that 
 $\til S_1, \til S_2$ are a \emph{product pair} of elevations.
 \end{defn}
 
We emphasize that for $\til S_1, \til S_2$ a product pair, the nearest-point projection of one onto the other lies in a uniformly bounded neighborhood of $\til{\Sigma}\times \{0\} $ for a  maximal essential proper  subsurface $\Sigma$.\\
 
\noindent {\bf Case 2: Annular connectors.}\\
Recall that $M_r$ is obtained from $S\times I$ after drilling finitely
many curves $\{\sigma_i\}$. For any $z \in \sigma_i \times \{0,1\}$, for some $i$, suppose that $z_I= z\times I \setminus \, {\rm{Int}} (N_\ep (\sigma_i))\subset  S\times I$ is contained in $M_r$, i.e.\ $z_I$ does not intersect any $N_\ep (\sigma_j), j \neq i$. Let $z_I^\pm$ denote the two components of $z_I$, and let $m_z$ denote the meridian of $N_\ep (\sigma_i)$
passing through $z_I^\pm \cap \partial N_\ep (\sigma_i)$. Then an \emph{annular connector} $\eta$ in $M_r$ starts and ends at $\{z\}  \times \{0,1\}$, 
traverses a connected component of $z_I$,  wraps around 
$m_z$ finitely many times and finally traverses a (possibly same) connected component of $z_I$. 
 If $\eta$  starts and ends at the same point $\{z\}  \times \{0\}$ (or $\{z\}  \times \{1\}$), then it wraps around 
$m_z$ (the meridian) $k$ times for some $k \in \Z$; else if it starts and ends at distinct points
in $\{z\}  \times \{0,1\}$, it wraps around 
$m_z$  (the meridian)  $(k+\half)-$ times for some $k \in \Z$. Any elevation of an annular connector in $M_r$ is called \emph{an annular connector in $\til M_r$.}
If the number of times an annular connector
wraps around 
$m_z$ does not equal $\pm \half$, we call it 
and its elevations  \emph{strict annular} connectors. Let $\til S_1, \til S_2$ denote the elevations of the boundary components of $M_r$ passing through the end-points of an annular connector $\til \eta$. We observe the following:

\begin{lemma}\label{lem-nppline}
Suppose that $\til \eta$ is a strict annular connector between $\til S_1, \til S_2$ joining $\til z_1 \in \til S_1$ to 
$\til z_2 \in \til S_2$. Let $\til\sigma, \til \sigma'$ denote the elevations of $\sigma_i$ through $\til z_1 $ and $\til z_2$ respectively. Then $\pi_{12}(\til S_1)$ (resp.\ $\pi_{21}(\til S_2)$) lies in a uniformly bounded neighborhood of
$\til \sigma'$ (resp.\ $\til\sigma$).
\end{lemma}

\begin{proof}
Let $z_1$ (the image of $\til z_1$ under the covering projection) be the base-point for $\pi_1(M_r)$, and let $z_\sigma$ be an annular connector that 
wraps around 
$m_z$  only half a time. Identify $\pi_1(S_2)$ with loops based at $z_1$ preceded and succeeded by $z_\sigma$ with appropriate orientations.

Then the result follows from the fact that $\pi_1(S_l) \cap \pi_1(S_k)^g = \Z$ for
$g \in \pi_1(M_r)$ representing any strict annular connector starting and ending 
at $z_1$, and $1 \leq l, k \leq 2$.
\end{proof}

\begin{defn}\label{defn-annularpair}
	If  $\til S_1, \til S_2$ are connected by a strict annular connector, we say that 
	$\til S_1, \til S_2$ are an \emph{annular pair} of elevations.
\end{defn}
 
\noindent {\bf Case 3: Cobounded connectors.}
\begin{defn}\label{def-cobpair}
	If  $\til S_1, \til S_2$ is a pair of distinct elevations that are neither 
a	product pair, nor an annular pair, we say that 
$\til S_1, \til S_2$ are a \emph{cobounded pair} of elevations.
\end{defn}

The terminology is justified by the following.

\begin{lemma}\label{lem-cobpair}
If $\til S_1, \til S_2$ are a cobounded pair  of elevations, then $\pi_{12}(\til S_1)$ is uniformly bounded in diameter.
\end{lemma}
 
\begin{proof}
Let $H_i < \pi_1(M_r)$ denote the stabilizers of $\til S_i,$ 
$i=1,2$. The Lemma then follows from the fact that $H_1 \cap H_2 =\{1\}$ (see, for instance,   the proof of  \cite[Corollary 3]{swarsuss}).
\end{proof}

\subsubsection{Preferred extended connectors in drilled atoms}\label{sec-prefconnatom} $ $ \\
We now describe a preferred family of quasigeodesics in $\til M_r$ connecting
$x \in \til S_1$ to $y \in \til S_2$. 
Denote $\pi_{12}{\til S_1} $ by $Z_2$ and $\pi_{21}{\til S_2} $ by $Z_1$. Then there are the following preferred product regions in $\til M_r$:

\begin{enumerate}
\item If  $\til S_1, \til S_2$ are  a product pair, then there exists 
 a proper essential subsurface $\Sigma$ of $S$  and elevations $\til \Sigma_i$, $i=1,2$  of boundary components of the product region $\Sigma \times [0,1]$ such that 
$Z_1 $ (resp.\ $Z_2$) is coarsely  $\til \Sigma_1$ (resp.  $\til \Sigma_2$). Also
$Z_i$ are (coarsely) the boundary components of $\til\Sigma \times [0,1]$, where
the $[0,1]$ direction has uniformly bounded length. We normalize the length of
$z \times [0,1]$ to one for all $z \in \til\Sigma$.
\item If  $\til S_1, \til S_2$ is an  annular  pair, then  $Z_1, Z_2$ are coarsely elevations
of (curves parallel to) $\sigma_i$. Further,  $Z_1, Z_2$ are boundary components
of a flat strip $\til \sigma_i \times [0,k\, l(m_z)+a]$, where $k$ equals  the number of times the annular connector wraps around the meridian $m_z$, 
$l(m_z)$ denotes the length of the meridian $m_z$, and $a$ equals the sum of the lengths of $z_I^\pm$, assuming without loss of generality that they are equal. 
\item If  $\til S_1, \til S_2$ is a cobounded  pair, then  $Z_1, Z_2$ are coarsely points, i.e.\  $Z_1, Z_2$ have uniformly bounded diameter.
\end{enumerate}

In all three cases, $\til S_i$ is strongly hyperbolic relative to $Z_i$, $i=1,2$.
Further, there is a natural product  $Z \times [a_1,a_2]$ embedded in $\til M_r$
	with $Z \times \{a_i\}=Z_i$, $i=1,2$. The interval $[a_1,a_2]$ has the following properties:
\begin{enumerate}
\item If  $\til S_1, \til S_2$ are  a product pair, $[a_1,a_2]$ has length one,
\item If  $\til S_1, \til S_2$ is an  annular  pair, $[a_1,a_2]$ has length equal to $k\, l(m_z)+a$,
\item If  $\til S_1, \til S_2$ is a cobounded  pair,  $[a_1,a_2]$ has length equal
to the length of an electro-ambient geodesic in $\til M_r$ joining $\til z_1$
and $\til z_2$.
\end{enumerate}

Next, for any $x_i\in \til{S_i}$, $i=1,2$, let $y_i \in Z_i$ denote a nearest point projection (in the intrinsic metric on $\til{S_i}$) of $x_i$ onto $Z_i$.
Identifying $Z_i$, $i=1,2$ with $Z \times \{0\} \subset Z \times [0,1]$
and  $Z \times \{1\} \subset Z \times [0,1]$ respectively, we have the following preferred family of paths joining $x_1, x_2$ in $\til M_r$:

 Projecting both $y_1, y_2$ to the $Z-$factor, we get points that we call $y_1, y_2$ again to avoid cluttering notation. Let $[y_1, y_2]$ denote the geodesic in $Z$ joining $y_1, y_2$. Let $p$ be any point on  $[y_1, y_2]$. Then the preferred collection of paths joining $x_1, x_2 \in \til M_r$ are given by the concatenation of the following segments:
 \begin{enumerate}
 \item the geodesic $[x_1,y_1] \subset \til S_1$ joining $x_1, y_1$,
 \item the geodesic $[y_1,p \times \{0\}] \subset \big( Z  \times \{0\}\big) (=Z_1) \subset \til S_1$ joining $y_1, p \times \{0\}$ in $\big( Z  \times \{0\}\big)$,
 \item the vertical interval $p \times [0,1] $ traveling from $ p \times \{0\}$ to  $p \times \{1\}$,
 \item the geodesic $[p \times \{1\},y_2] \subset \big( Z  \times \{1\}\big) (=Z_2) \subset \til S_2$ joining $p \times \{1\},y_2$ in $\big( Z  \times \{1\}\big)$,
 \item the geodesic $[y_2,x_2] \subset \til S_2$ joining $ y_2, x_2$.
 \end{enumerate}
 
 Note that there is only one vertical interval $p \times [0,1] $ traveling from $ p \times \{0\}$ to  $p \times \{1\}$ in each member of the family given above. Let $\FF(M_r, x_1, x_2)$ denote the above family. Elements of $\FF(M_r, x_1, x_2)$ for $x_1 \in \til{S_1}, x_2 \in \til S_2$ will be referred to as \emph{preferred extended connectors}.
 The construction above shows the following:
 
 \begin{lemma}\label{lem-preftracksea}
 Each $\alpha\in \FF(M_r, x_1, x_2)$ tracks the $\dpe-$geodesic and the electro-ambient quasigeodesic between $x_1, x_2$ in the intrinsic metric on
 $\til M_r$.
 \end{lemma}
 
 \begin{proof}
  The fact that $\alpha$ tracks the electro-ambient quasigeodesic
 \emph{ along elements of $\PP_r$ } follows from the fact that the nearest-point projection of
 any $P \in \PP_r$ onto $\til S_1$ is either uniformly bounded or (coarsely) an elevation of $\sigma_i$ to $\til S_1$. More precisely, in the second case, there exists an elevation $\til \sigma_i\subset \til S_1$, such that the nearest-point projection of
  $P \in \PP_r$ onto $\til S_1$ lies in a uniformly bounded neighborhood of 
  $\til \sigma_i$, and hence the concatenations $[x_1,y_1]\cup [y_1,p \times \{0\}]$, and $[p \times \{1\},y_2]\cup [y_2,x_2] $, used to define $\alpha$ have a maximal subpath each
  parallel to $P$. These are referred to as the subpaths of $\alpha$ \emph{along
  $P$}
  
  Away from elements of $\PP_r$, this is a consequence of Lemma~\ref{lem-atomshyp} and Lemma~\ref{lem-track}.
 \end{proof}

\subsubsection{Preferred extended connectors in concatenated drilled atoms}\label{sec-prefconnextdd} $ $ \\
Lemma~\ref{lem-preftracksea} can be extended to a 3-manifold obtained by concatenating finitely many atoms. As before, let $\Pi:\til F\to \TT$ denote  the Bass-Serre tree of $\til F=\til F_S$, with vertex spaces $X_v$ given by elevated atoms.

Corollary~\ref{cor-maln1} can be strengthened slightly as follows:

\begin{cor}\label{cor-maln2}
	There exists $L \geq 1$ such that the following holds.\\ Let $[a,b] \subset \TT$ denote a geodesic of length at least $L$ such that $X_a, X_b$ are elevated drilled atoms of $F_S$. Let $X_{[a,b]}$ denote the 3-manifold given by $\Pi^{-1}([a,b])$.
	Also, let $\til S_a\subset X_a$, and $ \til S_b\subset X_b$ denote  boundary components of $X_{[a,b]}$ such that $\til S_a, \til S_b$ do not separate 
	$X_{[a,b]}$ (i.e.\ $X_{[a,b]}$ lies entirely on one side of $\til S_a$; similarly for $\til S_b$). Then
	\begin{enumerate}
	\item  $X_{[a,b]}$ is strongly hyperbolic relative to the collection  $\PP_{[a,b]}$ (the elements of $\PP$ contained in it)
	\item  there exists $D$ depending only on
	$d_\TT (a,b)$ such that $\pi_{ab}(\til S_a)$ has diameter bounded by $D$
	(here, $\pi_{ab}$ denotes the nearest-point projection of $\til S_a$ on $\til S_b$).
	\end{enumerate}
\end{cor}
\begin{proof}
The first conclusion has the same proof as the first conclusion of Lemma~\ref{lem-atomshyp}. The second now follows Corollary~\ref{cor-maln1}.
\end{proof}

The construction of preferred extended connectors in $X_{[a,b]}$ can
now be carried out exactly 
as in Section~\ref{sec-prefconnatom}. We denote the family thus constructed as $\FF(
X_{[a,b]}, x_1, x_2)$ and refer to elements in 
$\FF(
X_{[a,b]}, x_1, x_2)$ as \emph{preferred extended connectors in $X_{[a,b]}$}. 
 Corollary~\ref{cor-maln2} gives us the following immediate consequence.
 \begin{lemma}\label{lem-cobddconnextdd}
 We continue with the setup of Corollary~\ref{cor-maln2}. Then, there exists $L \geq 1$ such that if $d_\TT (a,b) \geq L$, then preferred extended connectors
 $\alpha \in  \FF(
 X_{[a,b]}, x_1, x_2)$ are concatenations of three pieces, where the middle piece $\alpha_m$ is necessarily a cobounded connector, and the first and last ones are geodesics in $\til S_a, \til S_b$. 
 
 Thus, the end-points of $\alpha_m$  are  coarsely well-defined, i.e.\ there exists $D$ depends only on $d_\TT (a,b)$, and $z_1, z_2 $ in $\til S_a, \til S_b$ respectively
 such that the end-points of $\alpha_m$ lie on  $\til S_a, \til S_b$ at a distance
 of at most $D$ from $z_1, z_2 $.
 \end{lemma}
 
 Since  $X_{[a,b]}$ is strongly hyperbolic relative to the collection  $\PP_{[a,b]}$, as are all drilled atoms $\til M_c$ for $c \in [a,b]$, 
 the restriction of $\sigma \in  \FF(
 X_{[a,b]}, x_1, x_2)$ to  $\til M_c$ may be perturbed by a uniformly bounded amount, so that
 $(\sigma \cap \til M_c) \in \FF_c (M_c , x_1, x_2)$ for some $x_1, x_2 \in
 \partial  \til M_c$.

\begin{rmk}\label{rmk-cobddconstdependence}
We note that Lemma~\ref{lem-cobddconnextdd} implies that  $\alpha_m$ is a coarsely
well-defined electro-ambient quasigeodesic in $X_{[a,b]}$, and \emph{any preferred connector} between $\til S_a, \til S_b$ coarsely contains it. Note however that the parameter
$D$ determining coarseness depends on $d_\TT (a,b)$.
\end{rmk}

Let $[a,b] \subset \TT$ denote a geodesic of length at least $L$ such that $X_a, X_b$ are elevated drilled atoms of $F_S$. Let $X_{[a,b]}$ denote the 3-manifold given by $\Pi^{-1}([a,b])$. Assume further that the only drilled
atoms in $X_{[a,b]}$
are $X_a, X_b$. Let $\til S_a, \til S_b$ denote the boundary components of $X_{[a,b]} \setminus {\rm Int} (X_a \cup X_b)$ so that $\til S_a, \til S_b$ are 
also boundary components of $X_a, X_b$ respectively. We refer to $\til S_a, \til S_b$ as \emph{internal boundary components}    of $X_a, X_b$ respectively.
All other boundary components will be referred to as \emph{external boundary components}.

\begin{lemma}\label{lem-longundrilled}
Let $X_{[a,b]}$, $X_a, X_b$ be as above.
 Then there exists $L\geq 2$ such that if $d_\TT (a,b) \geq L$, the following holds. For any external boundary components $\til S_1, \til S_2$ of $X_{[a,b]}$,
there exists a coarsely well-defined geodesic $\lambda$ joining internal boundary components  $\til S_a, \til S_b$ in 
$X_{[a,b]} \setminus {\rm Int} (X_a \cup X_b)$ such that any $\alpha \in \FF$ connecting $\til S_1, \til S_2$  coarsely contains $\lambda$. Further, the coarseness is uniform,
independent of $X_{[a,b]}$.
\end{lemma}

\begin{proof}
Let $\pi_{1a}, \pi_{2b}$ denote nearest point projections of $\til S_1$ onto 
$\til S_a, $ and $\til S_2$ onto $ \til S_b$ respectively. Then the images of
$\pi_{1a}, \pi_{2b}$, given by $W_1, W_2$ respectively, are given by an elevation each of a proper essential subsurfaces of $\til S_a, \til S_b$ respectively. Hence, by Theorem~\ref{kld}, 
 there exists $L\geq 2$ such that if $d_\TT (a,b) \geq L$, 
 there exists a coarsely well-defined shortest geodesic $\sigma$ in $X_{[a,b]} \setminus {\rm Int} (X_a \cup X_b)$ joining  $W_1, W_2$. Any $\alpha \in \FF$ connecting $\til S_1, \til S_2$  must  coarsely join points in
$W_1, W_2$ and hence must coarsely  contain $\sigma$.
\end{proof}

\subsubsection{Controlling backtracking in elevated subbundles}\label{sec-baccktrackinsubbundle}$ $ \\ Let $\til M_r$ be a drilled
atom of $\til F$. Let $\til S_1, \til S_2, \til S_3$ denote three distinct
boundary components of $\til M_r$, and let $\BB$ be the elevation of a maximal undrilled subbundle of $F_S$ (Definition~\ref{def-atom}) such that $\BB \cap \til M_r=\til S_3$. 
Let  $\pi_{13}, \pi_{23}$ be nearest-point projections of $\til S_1, \til S_2$
onto $\til S_3$. Then the images of $\pi_{13}, \pi_{23}$ are given 
\begin{enumerate}
\item either by an elevation 
of a proper essential subsurface of $S_3$ (as in Section~\ref{sec-connatom}),
\item or is uniformly bounded in diameter.
\end{enumerate}
Further, by Theorem~\ref{kld}, $\pi_{13}(\til S_1), \pi_{23}(\til S_2)$ are uniformly quasiconvex in $\BB$, i.e.\ there exists $C \geq 1$ such that 
for any $\BB, \til S_1, \til S_2, \til{S_3}$ as above, $\pi_{13}(\til S_1), \pi_{23}(\til S_2)$ are $C-$quasiconvex in $\BB$. Hence, there exists $D_0$ such that if $d(\pi_{13}(\til S_1),\pi_{23}(\til S_2)) \geq D_0$, there is a coarsely
unique shortest path $\alpha$ 
 in $\BB$ joining $\pi_{13}(\til S_1), \pi_{23}(\til S_2)$. Thus, due to quasiconvexity of $\pi_{13}(\til S_1), \pi_{23}(\til S_2)$, there exists $C_1 \geq 0$ such that
\begin{enumerate}
\item either there is a coarsely well-defined geodesic $\alpha$ joining 
$\pi_{13}(\til S_1), \pi_{23}(\til S_2)$ in $\BB$. In this case, any $\dpe-$geodesic between $\til S_1, \til S_2$ in $\til M_r \cup \BB$
$C_1-$coarsely contains $\alpha$, i.e.\ 
any $\dpe-$geodesic between $\til S_1, \til S_2$ in $\til M_r \cup \BB$ contains a subpath tracking  $\alpha$ within distance $C_1$ of it. In such a case the 
$\dpe-$geodesic is said to \emph{have an allowable backtrack in $\BB$}.
 
\item or $\pi_{13}(\til S_1) \cup \pi_{23}(\til S_2)$ is $C_1-$quasiconvex in
$\BB$. In this case, a 
$\dpe-$geodesic $\lambda_p$ between $\til S_1, \til S_2$ is said to \emph{have  non-allowable backtracks in $\BB$} if $\lambda_p \cap \BB \neq \emptyset$.
By perturbing any such $\dpe-$geodesic between $\til S_1, \til S_2$ by a bounded amount, we obtain a $\dpe-$quasigeodesic that does not intersect
$\BB$ at all.
\end{enumerate}

\begin{rmk}\label{rmk-cobddconstdependencesubb}
Note that the constant $C_1$ above may be chosen to be uniform, i.e.\ independent
of the choice of $\BB, \til M_r, \til S_1, \til S_2, \til S_3$ (as there are only finitely many such possibilities up to the action of $\pi_1(F)$).
\end{rmk}

\subsubsection{Controlling small connectors in elevated subbundles}\label{sec-smallbaccktrackinsubbundle}$ $ \\
We generalize the above discussion in Section~\ref{sec-baccktrackinsubbundle} to the case where there are two elevations $\til S_3, \til S_3'$ coming from \emph{different} elevated drilled atoms abutting a common elevated maximal subbundle $\BB$.

Let $\til M_r^1 \neq \til M_r^2$ be  drilled
atoms of $\til F$. Let $\BB$ be  the elevation of a maximal undrilled subbundle of $F_S$ (Definition~\ref{def-atom}) such that $\BB \cap \til M_r^i=\til S_i$,
$i=1,2$. 
Let $\til S_i' \subset \til M_r^i$, $i=1,2$ denote 
boundary components of $\til M_r$ (different from $\til S_i$).
Let  $\pi_{11}, \pi_{22}$ be nearest-point projections of $\til S_1', \til S_2'$
onto $\til S_1, \til S_2$ respectively. Then the images of $\pi_{ii}$,
$i=1,2$ is given 
\begin{enumerate}
	\item either by an elevation 
	of a proper essential subsurface of $S_i$ (as in Section~\ref{sec-connatom}),
	\item or is uniformly bounded in diameter.
\end{enumerate}
 By Theorem~\ref{kld}, $\pi_{11}(\til S_1'), \pi_{22}(\til S_2')$ are uniformly quasiconvex in $\BB$. As in Section~\ref{sec-baccktrackinsubbundle} above, there exists $C_1 \geq 0, D_1\geq 0$ such that
 \begin{enumerate}
 	\item either there is a coarsely well-defined geodesic $\alpha$ joining 
 	$\pi_{11}(\til S_1'), \pi_{22}(\til S_2')$ in $\BB$. In this case, any $\dpe-$geodesic between $\til S_1', \til S_2'$ in $\til M_r \cup \BB$
 	$C_1-$coarsely contains $\alpha$.

 	\item or $\pi_{11}(\til S_1') \cup \pi_{22}(\til S_2')$ is $C_1-$quasiconvex in
 	$\BB$. Further, the distance between $\pi_{11}(\til S_1'), \pi_{22}(\til S_2')$ in $\BB$ is at most $D_1$. 	
 \end{enumerate}
 
 We shall refer to paths $\alpha$ as in item (1) above as \emph{long connectors in undrilled elevations}. Further, the triple  $\til M_r^1, \BB, \til M_r^2$ will be called a \emph{long connector triple}.
 
 If, on the other hand $\pi_{11}(\til S_1') \cup \pi_{22}(\til S_2')$ is $C_1-$quasiconvex in
 $\BB$ as in item (2) above, then the triple  $\til M_r^1, \BB, \til M_r^2$ will be called a \emph{short connector triple}.

\subsubsection{Defining the path family}\label{sec-defpf}$ $ \\
We now define the family $\FF$ of paths that will feed into Theorem~\ref{thm-sisto}.
Recall that $(\til{F},\dpe)$ is hyperbolic by Theorem~\ref{thm-tilfdpelhyp}.
For $x, y \in \til F$, define $\eta_p(x,y)$ to be a geodesic in $(\til{F},\dpe)$.
Now, replace 
$\eta_p(x,y)$ by a $\dpe-$quasigeodesic that 
\begin{enumerate}
\item does not backtrack from drilled atoms
(in the sense of Definition~\ref{def-nobt})
as in Corollary~\ref{cor-nobt},
\item does not have non-allowable backtracks in elevations of maximal undrilled
subbundles.
\end{enumerate}
 Abusing notation slightly, we continue to denote the
non-backtracking $\dpe-$quasigeodesic by $\eta_p$.
For each drilled atom $\til M_r$, and each connected component $\zeta_r$ of
$\eta_p\cap \til M_r$, 
replace  $\zeta_r$  by an electro-ambient geodesic joining its end-points. Denote the resulting path in $\til F$ joining the end-points of $\eta_p$ by $\eta$.

 We refer
to $\eta$ as the electro-ambient path obtained from $\eta_p$ by \emph{de-electrification}.
Next, define $\FF_p$ to be a collection $\{\eta_p(x,y)\}$ of uniform
 quasigeodesics in $(\til{F},\dpe)$ without backtracking in drilled atoms, one for every pair $x, y \in \til F$.
Define $\FF$ to be the collection $\{\eta(x,y)\}$ obtained from the collection 
$\{\eta_p(x,y)\}$ by de-electrification. For each such $\eta(x,y)$ define $\theta(x,y)$
to be the closure of $\eta(x,y) \setminus \bigcup_{P \in \PP} P$. Thus, 
$\theta(x,y)$ is obtained from $\eta(x,y)$ by removing the interiors of the intersections with elements of $\PP$.

In short, $\FF$ is obtained from geodesics in $(\til{F},\dpe)$ by
\begin{enumerate}
\item  removing backtracking in drilled atoms, and 
  removing non-allowable backtracks in elevations of maximal undrilled subbundles
\item subsequent de-electrification.
\end{enumerate}

\begin{rmk}\label{rmk-nobt}
The purpose of replacing a $\dpe-$geodesic by a  $\dpe-$quasigeodesic that does not backtrack from drilled atoms is to minimize intersections with elevated singular fibers
$\til S_x$ in $\til F$. This ensures that the only allowable backtracking is in elevations of maximal undrilled subbundles (see Definition~\ref{def-atom}).
\end{rmk}

\subsection{Stability}\label{sec-stab} The aim of this subsection is to prove the following stability condition, which is the main technical result of this section:

\begin{prop}\label{prop-stab} Given $D > 0$, there exists $C > 0$ such that the following holds.\\
	Let $\eta (x,y), \eta (u,v) \in \FF$ such $d(x, u) \leq D$ and $d(y, v) \leq D$.
Then  $\eta (x,y), \eta (u,v)$ track each other in a $C-$neighborhood of each other.
\end{prop}

We first observe a version of Proposition~\ref{prop-stab} in atoms.
\begin{lemma} Given $D > 0$, there exists $C > 0$ such that the following holds.\\
	Let $\eta (x,y), \eta (u,v) \in \FF$ be  such that
	\begin{enumerate}
	\item There exists an atom $\til{\mathbb{M}}$ equal to $\til M_r$ (drilled) or $\til M$ (undrilled) such that
	$\eta (x,y), \eta (u,v)$ are contained in $\til{\mathbb{M}}$.
	\item $x, u \in \til{S_x}$ and $y, v \in \til{S_y}$, where $\til{S_x}, \til{S_y}$
	are elevations of singular fibers, or equivalently, boundary components of
	$\til{\mathbb{M}}$.
	\item   $d(x, u) \leq D$ and $d(y, v) \leq D$.
	\end{enumerate}
	
	Then $\eta (x,y), \eta (u,v)$ track each other  in a $C-$neighborhood of each other.
\end{lemma}

\begin{proof}
This is a consequence of
 Lemma~\ref{lem-track} which guarantees that 
$\eta (x,y), \eta (u,v)$ track each other away from the elements of $\PP$ in  $\til{\mathbb{M}}$.
Further, since each element of $\PP$ is a flat $\R^2$,  geodesics in each element
of $\PP$ track each other provided they start and end nearby. The Lemma now follows from the construction of electro-ambient quasigeodesics.
\end{proof}

Before starting with the proof of Proposition~\ref{prop-stab}, we point out that the  main idea below
is to divide an element $\eta \in \FF$ into pieces that satisfy the property that its end-points are coarsely well-defined. 

\begin{proof}[Proof of Proposition~\ref{prop-stab}]
Assume, without loss of generality, that $x, u \in \til{S_1}\subset X_a$ and $y, v \in \til{S_2} \subset X_b$, where $a, b \in \TT$. There exists an indexing set
$\II$ giving a sequence of vertices $a=a_0, \cdots, a_n =b$, possibly with repetition, such that $\eta (x,y)$ traverses $\til M_{a_i}$ in order. 
Since $\eta (x,y) \in \FF$, it
\begin{enumerate}
\item does not backtrack in drilled atoms,
\item does not have non-allowable backtracks in  elevations of maximal undrilled subbundles.
\end{enumerate}

If there exists a maximal undrilled subbundle $\BB$ in which $\eta (x,y)$ has an
allowable backtrack (Section~\ref{sec-baccktrackinsubbundle}), we collect together all the (necessarily consecutive) vertices in the sequence $\{a_i\}$ that
correspond to atoms contained in $\BB$ and replace them by a single vertex $B_j$
for some $j$. We refer to such $B_j$ as an \emph{undrilled molecule}.
Thus, if some such $B_j$ occurs, then there exists $a_i \in  \II$ such that $B_j$ occurs in a unique subsequence of the form $a_iB_ja_i$ in $\II$.

Again, if there exists a maximal undrilled subbundle $\BB$ in which $\eta (x,y)$ has a long connector (Section~\ref{sec-smallbaccktrackinsubbundle}), then also 
we collect together all the (necessarily consecutive) vertices in the sequence $\{a_i\}$ that
correspond to (elevated) atoms contained in $\BB$ and replace them by a single vertex $B_k$
for some $k$. We also refer to such a $B_j$ as an \emph{undrilled molecule}.
Thus, if some such $B_k$ occurs, then there exists $a_i\neq a_s \in  \II$ such that $B_k$ occurs in a unique subsequence of the form $a_iB_ka_s$ in $\II$.

The construction of undrilled molecules now allows us obtain a new 
finite sequence
$\JJ= a_{1,1}, \cdots a_{1,m_1}, B_1, a_{2,1}, \cdots a_{2,m_2},B_2, \cdots$. Note that
the only possible repetition allowable in this  sequence are of the following form. If $\eta$ has an allowable backtrack in $B_j$ for some $j$, then there is a triple of
the form $a_{j,m_{j}},B_j, a_{j+1,1}$ with $a_{j,m_{j}} = a_{j+1,1}$ corresponding
to the same elevated drilled atom.

By the properties of an allowable backtrack (Section~\ref{sec-baccktrackinsubbundle}) or long connectors in undrilled elevations (Section~\ref{sec-smallbaccktrackinsubbundle}), 
both $\eta(x,y), \eta(u,v)$ coarsely contain a subpath $\alpha_j \subset B_j$ 
for all molecules $B_j$ (here we are conflating the index $B_j$ with the elevated 
maximal undrilled subbundle it indexes). Let $\til{S_j^-}, \til{S_j^+}$ denote
the boundary components of $B_j$ through which $\eta(x,y), \eta(u,v)$ enter
and leave $B_j$. Then there exist
\begin{enumerate}
\item $z_j^-(x,y) \in \eta(x,y) \cap \til{S_j^-}$, 
$z_j^+(x,y) \in \eta(x,y) \cap \til{S_j^+}$, 
\item $z_j^-(u,v) \in \eta(u,v) \cap \til{S_j^-}$, 
$z_j^+(u,v) \in \eta(u,v) \cap \til{S_j^+}$, 
\end{enumerate} 
such that
\begin{enumerate}
\item $z_j^-(x,y), z_j^-(u,v)$ lie at a distance of at most $2C_1$ from each other
on $\til{S_j^-}$,
\item $z_j^+(x,y), z_j^+(u,v)$ lie at a distance of at most $2C_1$ from each other
on $\til{S_j^+}$,
\item 
\end{enumerate}
Let $\eta_j(x,y)$ denote the subpath of $\eta(x,y)$ between $z_j^-(x,y),z_j^+(x,y)$. Let $\eta_j(u,v)$ denote the subpath of $\eta(u,v)$ between $z_j^-(u,v),z_j^+(u,v)$. 
By hyperbolicity of each $B_j$, $\eta_j(x,y)$ and $\eta_j(u,v)$ track each other
in a $C_1'-$neighborhood of each other, where $C_1'$ depends only on $C_1$ and
the hyperbolicity constant of $B_j$. Hence $C_1'$ is uniform.

To prove Proposition~\ref{prop-stab}, it therefore suffices to assume that the finite sequence $\JJ$ constructed from $\II$ \emph{does not contain any molecule $B_j$}. 

A caveat is in order. We note that if there is a short connector triple
$\til{M_r^1}, \BB, \til{M_r^2}$
as in Section~\ref{sec-smallbaccktrackinsubbundle}, then the atoms of the elevated maximal
undrilled subbundle $\BB$ are \emph{not} combined into a single molecule.
Further, there exists uniform $C_2 \geq 1$ (independent of $\BB$, $\eta(x,y), \eta(u,v)$) such that after a uniformly bounded perturbation if necessary, both $\eta(x,y), \eta (u,v)$ 
\begin{enumerate}
\item intersect
the same set of undrilled atoms $A_1, \cdots, A_k$ of $\BB$ in order without
backtracking in any of the atoms (i.e.\ after leaving any of the undrilled atoms $A_i,$  $\eta(x,y), \eta(u,v)$ do not return to it);
\item $k \leq C_2$ (this follows from the quasiconvexity property  of the union of projections used to define  short connector triples
 in Section~\ref{sec-smallbaccktrackinsubbundle}).
\end{enumerate}
We summarize this by saying that $\eta(x,y), \eta (u,v)$  have no backtracking 
in short connector triples.

We now return to the sequence of atoms  $\JJ = a_0, \cdots, a_m$ where 
\begin{enumerate}
\item $\JJ$ has no molecules, and hence
\item $\JJ$ is the vertex sequence of a \emph{geodesic} in the Bass-Serre tree
$\TT$ of $\til F$. 
\end{enumerate}
The absence of backtracking in short connector triples along with no backtracking in drilled atoms guarantees that $\JJ$ is the vertex sequence of a geodesic.

Let $L$ be the maximum of $C_2$ and the constants  in Lemma~\ref{lem-cobddconnextdd}
and Lemma~\ref{lem-longundrilled}. If there is a 
sequence of more than $L$ contiguous
undrilled blocks in $\JJ$, then choose a maximal  sequence $a_k, \cdots, a_l$
with $l-k \geq L$ indexing such blocks. Then, by Lemma~\ref{lem-longundrilled}, there exists a coarsely
well-defined $\alpha$ connecting boundary components $\til{S_k}, \til{S_l}$
(the \emph{internal boundary components} occurring in Lemma~\ref{lem-longundrilled}) coarsely contained in both $\eta(x,y), \eta (u,v)$.
Hence there exist $z_k(x,y) \in \eta(x,y)\cap \til{S_k}$, $z_k(u,v) \in \eta(u,v)\cap \til{S_k}$ such that $z_k(x,y), z_k(u,v)$ lie at a distance of at most $2C_1$ from each other
on $\til{S_k}$, where $C_1$ is the uniform constant of coarseness from 
Lemma~\ref{lem-longundrilled}. We may therefore assume henceforth that 
there does not exist a sequence of more than $L$ contiguous
undrilled blocks in $\JJ$.

Next, let $a_{k_1}, \cdots, a_{k_L}$ be a subsequence of $\JJ$ so that 
\begin{enumerate}
\item Each $a_{k_i}$ indexes a drilled atom,
\item $k_{i+1} > k_i$,
\item For any $i \in 1, \cdots, L$, every $j$ strictly between 
$a_{k_i}$ and $a_{k_{i+1}}$ indexes an undrilled atom,
\end{enumerate}
Then, by the simplification in the above paragraph $(a_{k_{i+1}}-a_{k_i})\leq L$.
Thus, any subsequence of $L$ drilled atoms in $\JJ$ interpolated 
only by undrilled atoms has length at most $L^2$. We shall refer to such a 
 subsequence of $\JJ$ as a \emph{subsequence of successive $L$ drilled atoms}.
 Note that in a subsequence of successive $L$ drilled atoms, the 
 drilled atoms need not be contiguous.
 
For such a subsequence of successive $L$ drilled atoms, let $a=a_{k_1}, b= a_{k_L}$, 
and $X_{[a,b]}$,
 $\til S_a, \til S_b$ be as in Lemma~\ref{lem-cobddconnextdd}. Then, 
 by Lemma~\ref{lem-cobddconnextdd},
 there exists a coarsely
 well-defined $\alpha$ connecting boundary components $\til{S_a}, \til{S_b}$
coarsely contained in both $\eta(x,y), \eta (u,v)$.

Finally divide $\JJ$ into subsequences of successive $L$ drilled atoms as follows:
let $\JJ = a_0=a_{n_0}, \cdots, a_{n_1}, \cdots, a_{n_s}, \cdots, a_m$, such that
\begin{enumerate}
\item each  $a_{n_i}$, $i=1, \cdots, s$ is a drilled atom
\item The subsequence of $\JJ$ between $a_{n_i}$ and $a_{n_{i+1}}$ (both included)
 has exactly $L$ drilled atoms for $i=0, \cdots, s$.
 \end{enumerate}
 We also call the initial sequence $a_{n_0}, \cdots, a_{n_1}$ a subsequence of successive $L$ drilled atoms. For each such subsequence of successive $L$ drilled atoms, there exists a coarsely
 well-defined $\alpha$ connecting boundary components $\til{S_a}, \til{S_b}$
 as before, and
 coarsely contained in both $\eta(x,y), \eta (u,v)$.
 
Setting $c=a_0, d= a_{n_s}$, it follows that there exist $C'$ and $y',v' \in \til{S_d}$
and subpaths  $\eta(x,y'), \eta (u,v')$ of $\eta(x,y), \eta (u,v)$ respectively
\begin{enumerate}
\item starting at $x, u$ respectively,
\item ending at $y', v'$ respectively,
\end{enumerate}
such that
\begin{enumerate}
\item $y', v'$ lie at a distance of at most $C'$ from each other
on $\til{S_d}$,
\item  $\eta(x,y'), \eta (u,v')$ track each other
in a $C'-$neighborhood of each other,
\item $C'$ is  independent of $x,y,u,v, \JJ$.
\end{enumerate}

 Let $\eta(y',y), \eta(v',v)$ denote the subpaths of  $\eta(x,y), \eta (u,v)$ respectively
 \begin{enumerate}
 	\item starting at $y', v'$ respectively,
 	\item ending at $y, v$ respectively.
 \end{enumerate}
 Let $p=a_{n_s}, q=a_m$. Then $X_{[p,q]}$ is a 3-manifold given by a concatenation of at most
 $L^2$ atoms. The proof of Lemma~\ref{lem-atomshyp} now shows furnishes relative hyperbolicity for each such
 $X_{[p,q]}$. Since there are only finitely many possibilities, the constants
 of relative hyperbolicity are uniform.
 The tracking properties of $\eta(y',y), \eta(v',v)$ now follow from relative hyperbolicity. Combining this with the tracking properties of $\eta(x,y'), \eta (u,v')$ already established, Proposition~\ref{prop-stab} follows.
\end{proof}

\subsection{Checking conditions 1,2,4,5,6 of Theorem~\ref{thm-sisto}}\label{sec-12456} We shall now show that the family $\FF$ defined above satisfy the conditions of Theorem~\ref{thm-sisto}.

\noindent {\bf  Condition (1) of Theorem~\ref{thm-sisto}:} \\ After rescaling $F$ if necessary, we might as well assume that the distance between any singular fiber of $F$
and the boundary $\partial (N_\ep(\sigma_i))$ of the neighborhood of any drilled curve
is at least 4. Hence, if $d_X (x,y) \leq 2$, then the condition follows from strong relative hyperbolicity of $\til M_r$ where $M_r$ is any drilled atom in $F$
(Lemma~\ref{lem-track}).\\

\noindent {\bf  Condition (2) of Theorem~\ref{thm-sisto}} is a consequence of the proof of stability of elements of $\FF$, Proposition~\ref{prop-stab}. \\

\noindent {\bf  Condition (4) of Theorem~\ref{thm-sisto}} follows from the fact that
any element of $\FF$ starting on $P_1\in \PP$ and ending on an element $P_2 \neq 
P_1$ of $\PP$, necessarily has points in the complement of $\bigcup_{P \in \PP} P$.\\

\noindent {\bf  Condition (5) of Theorem~\ref{thm-sisto}} follows from Corollary~\ref{cor-maln2}. Indeed, for any $P_1, P_2 \in \PP$,
 there exist $[a_1,a_2] \in \TT$ such that
  $P_i \in X_{a_i}$, for $i=1,2$. The strong relative hyperbolicity of 
  $X_{[a_1,a_2]}$ relative to the collection of elements of $\PP$ contained in it (with constants depending only on $d_\TT(a_1, a_2)$) furnishes
  Condition (5).  \\

 \noindent {\bf  Condition (6) of Theorem~\ref{thm-sisto}} also follows from 
Corollary~\ref{cor-maln2}. Indeed, as in Condition (5) above, we can choose
 $[a_1,a_2] \in \TT$ such that $d_\TT(a_1, a_2) \leq 2k$, where $k$ is as in
 Condition (6). Then 
 strong relative hyperbolicity of 
 $X_{[a_1,a_2]}$ relative to the collection of elements of $\PP$ contained in it 
 furnishes the constant $K$ required by Condition (6).\\

\subsection{Thin Triangles in $\FF$}\label{sec-thintriangles} It remains to prove the thin triangle condition, i.e.\ 
Condition 3 of Theorem~\ref{thm-sisto}. Let $a, b, c \in \til F$.

 Let $\gamma_1^d, \gamma_2^d, \gamma_3^d$ be sides of a quasigeodesic triangle
 in $(\til F, \dpe)$  with vertices $a, b, c$ used for constructing elements
 of $\FF$. Let
  Let $\gamma_1, \gamma_2, \gamma_3 \in \FF$ denote the elements of the
   path family constructed from them. Let $a$
 (resp.\ $b, c$) be the vertex opposite $\gamma_1 $ (resp.\ $ \gamma_2, \gamma_3$).
 Let $z$ denote a centroid.
 \\
 
 \noindent {\bf Case 1: $z$ lies in an undrilled atom  $\til M$ in $\til F$.}\\
In this case, thinness of triangles follows  from stability,
Proposition~\ref{prop-stab}. Indeed $z$ lies close to each of $\gamma_1, \gamma_2, \gamma_3$ (in the usual unelectrified metric on $\til F$), as each $\til S_i$, and hence $\til M$ is properly embedded in $(\til F,\dpe)$ (Lemma~\ref{lem-fiberproper}).
Hence for  $a$ (or $b$, or $c$) there exist a pair of paths $ \sigma_2, \sigma_3$
(given by subpaths of $ \gamma_2, \gamma_3$ respectively) starting from $a$
and ending close to $z$ in $\til M$. By Proposition~\ref{prop-stab},  $ \sigma_2, \sigma_3$ track each other (with uniform constants). A similar argument
works for $b, c$ completing this case.\\

 \noindent {\bf Case 2: $z$ lies in a drilled atom $\til M_r$ in $\til F$.}\\ In this case, there are three boundary components
$\til S_x$, $\til S_y$, $\til S_z$, of $\til M_r$, and subpaths $\beta_1^d, \beta_2^d, \beta_3^d$
of $\gamma_1^d, \gamma_2^d, \gamma_3^d$ such that the following holds.
In $\til M_r$ equipped with $\dpe$ the
subpaths $\beta_1^d, \beta_2^d, \beta_3^d$ 
 pass  close to $z$. We assume that $\beta_1^d, \beta_2^d$ have one
end-point each on $\til S_z$, $\beta_3^d, \beta_1^d$ have one
end-point each on $\til S_y$, $\beta_2^d, \beta_3^d$ have one
end-point each on $\til S_x$.

 This gives a hexagon in $\til M_r$, where
the other three sides (other than $\beta_1^d, \beta_2^d, \beta_3^d$)
 are geodesics in  $\til S_x$, $\til S_y$, $\til S_z$, joining
the two intersection points of the segments. Call these sides the \emph{complementary 
geodesics}, and denote them by $\alpha_x , \alpha_y, \alpha_z$ respectively. Call the subpaths  $\beta_1^d, \beta_2^d, \beta_3^d$ the \emph{internal 
geodesics}.

Suppose that at least
 one of the complementary 
geodesics, say $\alpha_x$ is long.  Note that $\alpha_x\subset \til S_x$ joins $x_2, x_3$, where $x_2, x_3$ are respectively the
intersection points  of $\beta_2^d, \beta_3^d$  with $\til S_x$.
Let $p_2, p_3 \in \beta_2^d, \beta_3^d$ respectively be points that are $\delta-$close to $z$ in the $\dpe-$metric, where $\delta$ is a constant depending only on the hyperbolicity constant of $(\til F, \dpe)$ (Theorem~\ref{thm-tilfdpelhyp}). 
Since $\til S_x$ is quasi-isometrically embedded (Lemma~\ref{lem-atomshyp})
in $(\til M_r,\dpe)$, $\alpha_x$ is a quasi-geodesic in $(\til M_r,\dpe)$.
(The quasigeodesic constants are uniform as there are only finitely many possibilities for drilled atoms.)
Let $\pi_x$ denote a nearest-point
projection of  $(\til M_r,\dpe)$ onto $\alpha_x$. Let $\pi_x(p_i)=q_i$,
$i=2,3$.  Then the $\dpe-$length of
the subsegment $\alpha_x(q_2,q_3)$ of $\alpha_x$ joining $q_2, q_3$ is at most $ 8\delta$ 
(see \cite[Lemma 3.2]{mitra-trees} for instance for a proof of this standard fact).
Since $\alpha_x$ is a quasi-geodesic in $(\til M_r,\dpe)$, the length of
$\alpha_x(q_2,q_3)$ in the intrinsic metric on $\til S_x$ is at most $ C\delta$, where $C$ depends only on the uniform  quasigeodesic constant for $\alpha_x$,
and is therefore uniform.

Let $[x_2,q_2]$ (resp.\ $[x_3,q_3]$)
 denote the geodesic in $\til S_x$ joining $x_2,q_2$ (resp.\ $x_3,q_3$).
 Also, let $[q_2,p_2]_{pel}$ (resp.\  $[q_3,p_3]_{pel}$) denote the  $\dpe-$geodesic segments in $\til{M_r}$ joining $q_2,p_2$ (resp.\ $q_3,p_3$).
 Then
 \begin{enumerate}
 \item $\beta_2'=[x_2,q_2] \cup [q_2,p_2]_{pel}$ is a $\dpe-$geodesic with uniform constants
 (cf.\ \cite[Lemma 3.2]{mitra-trees}),
 \item  $\beta_3'=[x_3,q_3] \cup [q_3,p_3]_{pel}$ is a $\dpe-$geodesic with uniform constants.
 \end{enumerate}
 Let  $\beta_2$ (resp.\ $\beta_3$) denote the subpaths of $\beta_2^d$
 (resp.\  $\beta_3^d$)
 joining $x_2, p_2$ (resp. $x_3, p_3$). By relative hyperbolicity of $\til M_r$
 (Lemma~\ref{lem-atomshyp}), and the tracking properties in Lemma~\ref{lem-track},
there exists $q_2' \in \beta_2$ (resp.\ $q_3' \in \beta_3$) such that 
the distance (in the unelectrified metric $d$ on $\til M_r$) between $q_2, q_2'$
(resp.\ $q_3,q_3' $) is uniformly bounded by a uniform constant $D'$. Hence
$d(q_2',q_3')\leq C\delta + 2D'$.

Let $\gamma_2'$ (resp. $\gamma_3'$) be the subpath of $\gamma_2$ from $a$ to $q_2'$ (resp.\ $a$ to $q_3'$).  By Proposition~\ref{prop-stab}, $\gamma_2'$ and $\gamma_3'$ track each other within a uniform distance $C''$ of each other.

\begin{comment}
	content...

By 
 Replace the subpath of $\beta_2^d$ (resp.\ $\beta_3^d$) between    $x_2, p_2$
(resp.\ $x_3, p_3$) by the union of the $\dpe-$geodesic segments $[q_2,p_2]_{pel}
\cup[]$ (resp.\  $[q_3,p_3]_{pel}$ to obtain $\beta_2'$ (resp.\ $\beta_3'$). It is now easy to see
that $\beta_2',\beta_3'$ are $\dpe-$quasigeodesics with uniform constants 
\begin{enumerate}
\item starting at $q_1, q_2$ that are within a (non-electrified) distance at
most  $ C\delta$ from each other,
\item pass $\dpe-$ close to $z$,
\item the subpath of $\beta_2^d$ (resp.\ $\beta_3^d$) between    $x_2, p_2$
(resp.\ $x_3, p_3$) and electro-ambient representatives of $[x_2,q_2]\cup [q_2,p_2]_{pel}$ (resp.\ $[x_3,q_3]\cup [q_3,p_3]_{pel}$) track each other
by Lemma~\ref{lem-atomshyp} and Lemma~\ref{lem-track},
\item Let $\gamma_2'$ (resp. $\gamma_3'$) be the subpath of $\gamma_2$ from $a$ to $p_2$ (resp.\ $a$ to $p_3$), and then replacing $\beta_2^d$ (resp.\ $\beta_3^d$) 
as above to obtain $\gamma_2''$ (resp. $\gamma_3''$). Then electro-ambient representatives of $\gamma_2''$ and $\gamma_3''$
track each other as also elements of $\FF$ joining 
end-points of $\gamma_2''$ and $\gamma_3''$ track each other
\end{enumerate}
\end{comment}

Removing the initial subpath of $\beta_2^d$ (resp.\  $\beta_3^d$) between $x_2, q_2'$ (resp.\ $x_3, q_3'$), we can replace the complementary geodesic $\alpha_x$ 
by a geodesic $\alpha_x'$ of length  at most $ C\delta + 2D'$ joining $q_2',q_3'$.
Carrying out such replacements for all the long complementary 
geodesics, we obtain a $\dpe-$quasigeodesic hexagon whose complementary 
geodesics are uniformly bounded in length. Let $\beta_f^1, \beta_f^2, \beta_f^3$
be the internal geodesics of the resulting hexagon. By strong relative hyperbolicity of $\til M_r$ with respect to $\PP_r$, the internal geodesics 
$\beta_f^1, \beta_f^2, \beta_f^3$ satisfy Condition 3 of Theorem~\ref{thm-sisto}.

Combining this with the tracking properties of pairs such as $\gamma_2'$ and $\gamma_3'$ proved above, using Proposition~\ref{prop-stab},  it follows that $\gamma_1, \gamma_2, \gamma_3 \in \FF$  satisfy Condition 3 of Theorem~\ref{thm-sisto}.

\begin{proof}[Proof of Theorem~\ref{thm-relhyp}]
That the family $\FF$ constructed in Section~\ref{sec-defpf} satisfies Conditions
1,2,4,5,6 of Theorem~\ref{thm-sisto} has been established in Section~\ref{sec-12456}. The thin triangles condition, i.e.\ Condition
3 of Theorem~\ref{thm-sisto} has been checked above in this subsection. Hence,
by Theorem~\ref{thm-sisto}, $\til F$ is strongly hyperbolic relative to the collection $\PP$.
\end{proof}

\subsection{Relative quasiconvexity} Let $S_0$ be a drilled fiber in $E$.
Further, after isotoping the drilled curves if necessary, we  can assume that the collection  of  drilled curves $\sigma_1, \cdots, \sigma_k$ in $S_0$
is maximal, i.e.\ no other drilled curve may  be isotoped 
into $S_0$ in the complement of $\cup_{i=1, \cdots, k} N_\ep(\sigma_i)$.

\begin{defn}\label{def-reduced}
We say 
that the collection of drilled fibers is \emph{reduced} if the collection of 
drilled curves in any singular fiber is maximal in the above sense.
\end{defn}
 Then $S_0 \setminus (\bigcup_i \sigma_i)$ consists of finitely many components $K_1, \cdots, K_m$. By Lemma~\ref{kld}, we have the following:
 \begin{lemma}\label{lem-qc}
Each $\pi_1(K_i)$ is quasiconvex in $\pi_1(E)$. 
 \end{lemma}
 
 Next, we consider $F$ and $(\til F, \dpe)$. Then we have
 
 \begin{lemma}\label{lem-dpeqc}
 Let $K$ be a component of $S_0 \setminus (\bigcup_i \sigma_i)$ as above for $S_0$ a drilled fiber. Then	there exists $C \geq 1$ such that any elevation 
 $\til K$ is  quasiconvex $(\til F, \dpe)$.
 \end{lemma}
 
 \begin{proof}
  The same argument  as in Section~\ref{sec-hypofdpe} in the paragraph "Identifying $\rho-$thin annuli" identifies the collection of essential annuli with core curve homotopic to a curve in $K$. Corollary~\ref{cor-weakannuliflare}
  now establishes that $K$ flares in all directions in the sense of Definition~\ref{def-flareinall}. Hence, by Corollary~\ref{cor-qc} and Remark~\ref{rmk-qc}, 	there exists $C \geq 1$ such that any elevation 
  $\til K$ is  quasiconvex $(\til F, \dpe)$.
 \end{proof}

 We finally have the following:
 
 \begin{prop}\label{prop-relqc}
 Let $K$ be a component of $S_0 \setminus (\bigcup_i \sigma_i)$ as above for $S_0$ a drilled fiber. Then	there exists $C' \geq 1$ such that any elevation 
 $\til K$ is relatively $C'-$quasiconvex.
 \end{prop}
 
 \begin{proof} Let $\TT$ denote the Bass-Serre tree of $\til F$, 
 	$\Pi:\til F \to \TT$ denote the tree of spaces structure, and $v$ be the
 	vertex such that $\til S_0 \subset \til M_v$. Let $\TT_0$ denote the
 	$C-$neighborhood of $v_0$ in $\TT$, where $C$ is as in Lemma~\ref{lem-dpeqc}.
 	Also, let $X_0 = \Pi^{-1}(\TT_0)$.
 By Lemma~\ref{lem-dpeqc} the $\dpe-$geodesic joining a pair of points
 $x, y \in \til K$ lies in $X_0$. Further, the proof of Lemma~\ref{lem-fiberproper}  establishes that $\til K$ with its intrinsic metric
 is qi-embedded in $(X_0,\dpe)$.
 
The proof of Theorem~\ref{thm-relhyp} applied to $X_0$ establishes strong-relative hyperbolicity of $X_0$ relative to the collection $\PP_0$ consisting of the
elements of $\PP$ contained in $X_0$. Since  $\til K$ with its intrinsic metric
is qi-embedded in $(X_0,\dpe)$, the construction of the path family $\FF$ in Section~\ref{sec-defpf} shows that we can take $\eta(x,y)$ to lie on $\til K$
for $x, y \in \til K$. Hence, $\til K$ is relatively $C'-$quasiconvex for some $C'$.
 \end{proof}

 \begin{comment}
 	content...

The construction of the path  family $\sxy$ joining points on any elevation $\til K_i$
come from electro-ambient quasigeodesics, hence these are quasigeodesics by Lemma~\ref{lem-ea-drill}. In particular, they lie in a uniformly bounded neighborhood
of  $\til K_i$. Since Theorem~\ref{thm-sisto} and Theorem~\ref{thm-relhyp} show that the path family constructed consists of (relative) quasigeodesics, it follows that
$\pi_1(K_i)$ is relatively quasiconvex in $\pi_1(F)$ (relative to the parabolics $\PP$).
Then use this to conclude that geodesics
lie in a finite distance along $\TT$. Finally use relhyp of these truncations,
and the fact that the 3-manifold in question is geo fin.
 \end{comment}

\section{Cubulation}\label{sec-cub} We refer the reader to \cite{hw-gafa,wise-cbms,wise-hier} for details on virtually special CAT(0) cube-complexes. However, before attempting  to cubulate $G=\pi_1(F)$, we first describe another graph of groups structure on $G$.  

\subsection{Another graph of groups structure on $G$}\label{sec-gog} 
	We  construct a new graph of groups structure on $\pi_1(F)$, with a new underlying graph $\GG_0$ as follows. 
	
	Assume that the collection of drilled fibers is reduced in the sense of Definition~\ref{def-reduced}.
The vertex groups of $\GG_0$ 
	are all isomorphic to $H=\pi_1(S)$, and consist of the following:
	\begin{enumerate}
		\item For every singular fiber $S_v$, necessarily undrilled by construction, we have a vertex group $G_v$ isomorphic to $H$.
		\item For a drilled edge $e$, let $S_{x_1}, \cdots, S_{x_p}$ denote
		the drilled fibers where $x_i$ are points in order along $e$, where $e$
		has initial and final vertices $v_i, v_f$.
	  Interpolate vertices 
	 $v_i=v_0, v_1, \cdots, v_p =v_f$ where $v_i$ lies between $x_i, x_{i+1}$
	 for $1\leq i \leq p-1$. The points $x_1, \cdots, x_p$   will be referred to as \emph{special drilled points}, and $ v_1, \cdots, v_{p-1}$  will be referred to as \emph{interpolating points}.
\end{enumerate} 

\begin{comment}
	content...
	We describe the second class of vertex spaces more explicitly. Let $e=[0,1]$ be a drilled edge. Without loss of generality, suppose that the fibers over $\frac{i}{N}$, $0<i<N$, are the drilled fibers for some $N \in \natls$. Then  $\half(\frac{i}{N}+\frac{i+1}{N})$, $1\leq i \leq N-2$ 
	are the new vertices and each vertex group is  isomorphic to $H=\pi_1(S)$. 
	The points $\frac{i}{N}$, $0<i<N$
\end{comment}

The vertices of $\GG_0$ consist of the vertices of $\GG$ \emph{along with 
interpolating points}.

We now define the edge spaces. The intervals $[v_i, v_{i+1}]$  will be termed \emph{subdivision intervals}. Note that each subdivision interval contains a unique special drilled point $x_{i+1}$.
Let $\eta=[c,d]$ be a subdivision interval, so that the vertex groups $G_c, G_v$
are isomorphic to $H$. Now, restrict the drilled bundle $F$ to $\eta$ to obtain
$F_\eta$. Then $F_\eta$ is homeomorphic to $S \times [0,1]$ after removing 
$\ep-$neighborhoods of a non-empty 
family $\{\sigma_i\}$ of disjoint, homotopically distinct, essential simple closed curves on a unique drilled fiber $S_{x_i}$. Let $K_{\eta,i},  \, i=1, \cdots, m$ denote the components of $S \setminus \bigcup_i \sigma_i$. Then $\pi_1(F_\eta)$ admits a graph
of group description with two vertex groups $G_c, G_d$, each isomorphic to $H$; and $m$ edge groups, isomorphic to $\pi_1(K_{\eta,i}), \, i=1, \cdots, m$. Let $\{e_{\eta,i}, i=1,\cdots, m\}$ denote the resulting edges `living over' 
the subdivision interval $\eta$.
We call these \emph{subdivision edges}. The groups $\pi_1(K_{\eta,i}), \, i=1, \cdots, m$ are called the  \emph{subdivision edge groups}.
The edges of $\GG_0$ are obtained as follows.
\begin{enumerate}
\item  If there is an edge of $\GG$ that has no special drilled points, leave it as it is in $\GG_0$. We refer to these as \emph{undrilled edges} of $\GG_0$.
\item Next, replace each interval with at least one  special drilled point by subdivision intervals, and finally each 
subdivision interval by the subdivision edges that live over it.
\end{enumerate} 

Note that the above discussion goes through even when the initial and final vertices of the drilled edge $e$ coincide,
so that there is a monodromy map $\phi$ for the bundle $E$ restricted to the (closed) $e$. 
We explicate the inclusion maps for $\pi_1(K_{\eta,i}), \, i=1, \cdots, m$ into the vertex groups $H_i, H_f$ corresponding to the vertices $v_i, v_f$ (the initial and final vertices of $e$). We identify $\pi_1(K_{\eta,i}) \, i=1, \cdots, m$ with subgroups of $H_i$ via the product structure on $S\times [0,d)$. Then, modulo this identification, the edge-to-vertex group maps for $\pi_1(K_{\eta,i}), \, i=1, \cdots, m$ into $G_c$ is the inclusion. The same holds for $d \neq v_f$.
For $v_f=d$, the  edge-to-vertex group maps for $\pi_1(K_{\eta,i}), \, i=1, \cdots, m$ into $G_d$ is given by inclusion followed by $\phi_\ast$, the map induced by the monodromy $\phi$.

It remains to identify the edges and edge groups for the underlying graph $\GG_0$. The edges are of exactly two kinds:
 undrilled edges, and subdivision edges. The edge groups for undrilled edges correspond to $H$. The edge groups for subdivision edges correspond to the  subdivision edge groups $\pi_1(K_{\eta,i}), \, i=1, \cdots, m$.
Finally, the edge-to-vertex group maps are given as above.

\begin{defn}\label{def-reducedgraph}
The graph $\GG_0$ obtained as above will be called the \emph{reduced form} of $\GG$
for $F$.
\end{defn}

\begin{defn}\label{def-component}
A maximal connected  undrilled subgraph $\KK$ of $\GG$ will be called an \emph{undrilled
	component of $\GG$}, and  $\FF_\KK:=\Pi_0^{-1} (\KK)$ will be called  an \emph{undrilled
	constituent of $\FF$}. 
\end{defn}

We finally modify $\GG_0$ to another graph $\GG_\K$ by collapsing each undrilled
component $\KK$ of $\GG$ to a single vertex, i.e.\ $\GG_\K$  is the quotient space
obtained from $\GG_0$ under the equivalence relation $x\sim y$ if and only if
$x, y $ belong to the same undrilled
component $\KK$ of $\GG$. Let $v_{\KK}$ be the resulting vertex
of $\GG_\KK$. We refer to $v_{\KK}$ as the \emph{undrilled
	component vertex of $\GG_\K$ associated with $\KK$.}
	Thus, the set of  vertices of $\GG_\K$ is precisely the collection 
	$\{v_{\KK}\}$ of undrilled
	component vertices.
	
The vertex space associated to $v_{\KK}$ is then declared to be
$\FF_\KK$. The edge to vertex inclusions are given by the composition of
\begin{enumerate}
\item edge to vertex inclusion maps over $\GG_0$ composed with  
\item  inclusions of vertex spaces over $v\in \GG_0$  to the 
vertex spaces over $v_\KK\in \GG_\K$, where  $\KK$ is the undrilled
component  of $\GG$ containing $v$.
\end{enumerate}

\begin{defn}\label{def-canred}
$\GG_\KK$ will be called the \emph{canonical reduction} of $\GG$
for $F$.
\end{defn}
We conclude this subsection with the following observation that follows from
the above construction.

\begin{lemma}\label{lem-edgesofcanonical}
The edge spaces of $\GG_\KK$ are precisely the components   of $S_d \setminus \bigcup_i \sigma_i$, where $S_d$ ranges over a reduced collection of  drilled fibers .
\end{lemma}

\subsection{Cubulating drilled bundles}
We shall need the following theorem due to Wise:
\begin{theorem}\cite[Theorem 15.1]{wise-hier}\label{15.1}
Let \(G\) be a group satisfying the following:
	\begin{enumerate}
		\item \(G\) is hyperbolic relative to virtually abelian subgroups.
		\item \(G\) splits as a graph of groups \(\Gamma\) where each edge group is relatively quasiconvex in \(G\).
		\item Each vertex group is virtually special.
		\item For each edge \(e\), the edge group \(G_e\) has trivial intersection with each \(Z^2\) in the fundamental group of the graph of groups \(\Gamma - e\).
	\end{enumerate}
	Then \(G\) is the fundamental group of a virtually special cube complex.
\end{theorem}

\begin{comment}
	content...

Consider a hyperbolic group \(G_0 = G_1 *_H G_2\) coming from the graph of spaces \(M\) which is an \(S\)-bundle over the graph of two loops joined by an edge \(e\). Here \(S\) is a hyperbolic surface with \(H = \pi_1(S)\), and \(G_1, G_2\) are fundamental groups of two closed hyperbolic 3-manifolds fibering over the circle with \(S\) as fibers. Pick a separating essential simple closed curve \(\gamma\) in \(S_e\), where \(S_e\) is the fiber over the midpoint of \(e\). Let \(M_{\text{drilled}}\) be \(M\) with a tubular neighborhood of \(\gamma\) removed. Let \(S_1, S_2\) be the subsurfaces of \(S_e\) appearing as components of \(S_e - \gamma\).

 The group \(G = \pi_1(F)\) admits the following graph of groups structure. The underlying graph \(\Gamma\) has two vertices \(v_1, v_2\) with two edges \(e_1, e_2\) in common between them. The vertex groups are \(G_1, G_2\) as before. The edge groups correspond to \(\pi_1(S_1)\) and \(\pi_1(S_2)\), with the inclusions being \(\pi_1(S_i)\hookrightarrow \pi_1(S_e)\hookrightarrow G_j\) for \(i, j= 1,2\).
\end{comment}

For us, $G=\pi_1(F)$, and let $\Pi_0: F \to \GG$ denote the natural projection.

\begin{theorem}\label{thm-cub} Suppose that the graph $\GG_\K$ is the canonical reduced form of $\GG$. Suppose further that for every undrilled
	component $\KK$ of $\GG$, $\pi_1(\FF_\KK)$ is virtually special cubulable.
	Then the group \(G=\pi_1(F)\) is virtually special cubulable.
\end{theorem}

\begin{proof}
Theorem~\ref{thm-relhyp} 
proves that $\til F$ is strongly hyperbolic relative to $\PP$. Since the stabilizers of each $P \in \PP$ is $\Z+\Z$, Condition (1) of Theorem~\ref{15.1} is satisfied.\\

By Lemma~\ref{lem-edgesofcanonical}, the edge spaces over $\GG_\KK$ are given by
the components   of $S_d \setminus \bigcup_i \sigma_i$, where $S_d$ ranges over a reduced collection of  drilled fibers. The fundamental groups of these components
are relatively quasiconvex in $G$ by Proposition~\ref{prop-relqc}. Hence 
Condition (2) of Theorem~\ref{15.1} is satisfied.\\

Condition  (3) of Theorem~\ref{15.1} follows from the hypothesis that 
$\pi_1(\FF_\KK)$ is virtually special cubulable.\\

Condition (4) of Theorem~\ref{15.1} follows from the hypothesis that 
$\GG_\K$ is the canonical reduced form of $\GG$. Indeed, this hypothesis guarantees
that there are no accidental parabolics in the components $K$  of $S_d \setminus \bigcup_i \sigma_i$, where $S_d$ ranges over a reduced collection of  drilled fibers, i.e.\ no non-peripheral essential curve in any $K$  is freely homotopic to
a drilled curve. \\

Hence, by Theorem~\ref{15.1}, $G$ is virtually special cubulable.

\end{proof}

\subsection{Examples} We now give examples of surface bundles $E$ over graphs $\GG$ and drilled curves such that the hypotheses of Theorem~\ref{thm-cub} are satisfied:\\
\begin{eg}\label{eg-contratcible}
Each edge of $\GG$ contains a drilled surface. In this case, the undrilled components $\KK$ are precisely the vertices of $\GG$, and the associated spaces
are given by the fiber $S$. Since $\pi_1(S)$ is special cubulable, 
the hypotheses of Theorem~\ref{thm-cub} are satisfied.
\end{eg}

\begin{eg}\label{eg-3mfld}
 Undrilled components $\KK$ are either non-self-intersecting loops in $\GG$ or vertices. The associated vertex spaces are either hyperbolic 3-manifolds $M$ fibering over the circle, or the fiber $S$. By \cite{agol-vh}, $\pi_1(M)$ is virtually special and hence the hypotheses of Theorem~\ref{thm-cub} are satisfied.
 More generally, undrilled components $\KK$ could be homotopy equivalent to circles or contractible.
\end{eg}

A simple example for Example~\ref{eg-3mfld} is given by a graph $\GG$ with two 
vertices $v_1, v_2$, an edge $[v_1, v_2]$ and a loop at each vertex $v_1, v_2$.
Further, the edge $[v_1, v_2]$ has a single drilled fiber $S_w$ with one simple closed curve $\sigma \subset S_w$ drilled.

\begin{eg}\label{eg-mms}
Undrilled components $\KK$ are homotopy equivalent to a wedge of circles, and the
restriction of $E$ over any such $\KK$  are examples from \cite{mmscubulating}.
The main theorem of \cite{mmscubulating} then guarantees that the hypotheses of Theorem~\ref{thm-cub} are satisfied.

More generally, components $\KK$ could be a mixture of these cases, i.e.\ they could be
\begin{enumerate}
\item contractible, in which case the associated vertex space is homotopy equivalent to $S$,
\item homotopy equivalent to a circle, in which case the associated vertex space is homotopy equivalent to hyperbolic 3-manifolds $M$ fibering over the circle,
\item homotopy equivalent to a  wedge of circles, with the associated vertex space  homotopy equivalent to  one of the examples from \cite{mmscubulating}.
\end{enumerate}
\end{eg}

\section{Virtual algebraic fibering}\label{sec-algfib}

\begin{defn}
A finitely generated  group $G$ is said to \emph{ virtually algebraically fiber}
if there exists a finite index subgroup $G_1$ of $G$ such that $G_1$  
admits
a surjective homomorphism to $\Z$ with finitely generated kernel.
\end{defn}
A
 theorem of Kielak \cite{kielak} gives the following criterion for 
 virtual algebraic fibering.

\begin{theorem}\cite{kielak}\label{kielak}
	Let $G$ admit a geometric  action on a CAT(0) cube complex. Then the following are equivalent:
	\begin{enumerate}
	\item $G$ virtually algebraically fibers,
	\item the first $\ell^2-$betti number $\beta_1^{(2)} (G)$ equals zero.
	\end{enumerate}
\end{theorem}

\subsection{Vanishing first $l^2$ betti number} The aim of this subsection is to show:

\begin{prop}\label{prop-l2}
Let $F$ be a drilled surface bundle over a finite graph $\GG$. Let
$G=\pi_1(F)$. Then $\beta_1^{(2)} (G)=0$.
\end{prop}

As an immediate consequence of Proposition~\ref{prop-l2} and Theorem~\ref{kielak}, we have the following:

\begin{theorem}\label{thm-vfib}
Let $F$ be a drilled surface bundle over a finite graph $\GG$ satisfying the hypotheses of Theorem~\ref{thm-cub}. Then $G=\pi_1(F)$ virtually algebraically fibers. 
\end{theorem} 

To prove Proposition~\ref{prop-l2}, we shall need a couple of results.
A fundamental theorem of Lott-Lueck gives the following.
\begin{theorem}\cite[Theorem 0.1]{lott-luecke}\label{thm-ll}
Let $M_r$ be a drilled atom of $F$, and $H_r=\pi_1(M_r)$. Then 
 $\beta_1^{(2)} (H_r)=-\chi(M_r)$.
\end{theorem}

We shall also use the following theorem of Fernos-Valette:

\begin{theorem}\cite[Theorem 1.1]{fv-hha}\label{thm-fv}
Let $G$ be a graph of groups with at least one edge, such that all
vertex groups satisfy $\beta_1^{(2)} (G_v)=0$. If every edge group is infinite,
and for every edge-to-vertex group inclusion, the  edge group is of infinite index in the vertex group, then $\beta_1^{(2)} (G)=0$.
\end{theorem}

Recall from Section~\ref{sec-gog} that if $\eta$ is a subdivision interval, the restriction $F_\eta$ of $F$ to $\eta$ has a simple topology: $F_\eta$
 is homeomorphic to $S \times [0,1]$ after removing 
$\ep-$neighborhoods of a non-empty 
family $\{\sigma_i\}$ of disjoint, homotopically distinct, essential simple closed curves on a unique drilled fiber $S_x$. In this subsection, we shall refer to such
an $F_\eta$ as an \emph{ elementary drilled atom}. 

Let $\sigma_1, \cdots, \sigma_k \subset S_x$ denote the drilled curves. Let
$\T_1,\cdots, \T_k$ denote $k-$copies of the standard torus $S^1 \times S^1$. 
Also,
let $K_\eta$ denote the 2-complex obtained as a quotient space
of $S_x \sqcup \bigcup_i \T_i$ by identifying $S^1 \times \{0\} \subset \T_i$
with $\sigma_i$ via a homeomorphism. 
Then $F_\eta$ is homotopy equivalent to $K_\eta$. Further, $\pi_1(F_\eta)=
\pi_1(K_\eta)$ has a graph of groups description, where the underlying graph
$\GG$ has $k+1$ vertices $0, \cdots, k$ say, and
\begin{enumerate}
\item $\GG$ is a tree with one root vertex $0$, and all other vertices 
$1, \cdots, k$ are connected to $0$ by an edge each,
\item the vertex group $G_0 = \pi_1(S_x)$,
\item for $i=1, \cdots, k$, each vertex group $G_i =\Z\oplus \Z$,
\item each edge group is $\Z$.
\end{enumerate} 

\begin{lemma}\label{lem-induction}
Let $L$ be a group with $\beta_1^{(2)} (L)=0$. Let $H\subset L_1$ be a subgroup isomorphic to $\pi_1(S_x)$. Let $G=L*_H \pi_1(K_\eta)$, where $H$ is identified via
an automorphism with $G_0 = \pi_1(S_x)$. Then $\beta_1^{(2)} (G)=0$.
\end{lemma}

\begin{proof}
$G$ admits a graph of group description, where the underlying graph
$\GG$ has $k+1$ vertices $0, \cdots, k$ say, and
\begin{enumerate}
	\item $\GG$ is a tree with one root vertex $0$, and all other vertices 
	$1, \cdots, k$ are connected to $0$ by an edge each,
	\item the vertex group $G_0 = L$,
	\item for $i=1, \cdots, k$, each vertex group $G_i =\Z\oplus \Z$,
	\item each edge group is $\Z$.
\end{enumerate} 
Then $G$ satisfies the hypotheses of Theorem~\ref{thm-fv}. Hence, $\beta_1^{(2)} (G)=0$.
\end{proof}

\begin{lemma}\label{lem-l2}
	Let $F$ be a drilled surface bundle over a finite graph $\GG$
	such that $\GG$ is homotopy equivalent to a circle. Let
	$G=\pi_1(F)$. Then $\beta_1^{(2)} (G)=0$.
\end{lemma}

\begin{proof}
Let $\CC \subset \GG$ be a cycle with no repeated vertices. Then $\GG=\CC\cup \bigcup_{i=1,\cdots, k} \TT_i$, where
\begin{enumerate}
	\item  each $\TT_i$ is a finite tree 
\item $\TT_i$ intersects $\CC$ at a single point $p_i$,
\item $\TT_i \setminus \{p_i\}$ is disjoint from $\cup_{j \neq i} \TT_j$.
\end{enumerate} 
 Each edge of  $\TT_i$, $i=1,\cdots, k$,  can be subdivided as in Section~\ref{sec-gog}, so that each edge is a subdivision edge. In particular,
 after such subdivision, for every edge  $\eta$ of $\TT_i$, $i=1,\cdots, k$,
 $F_\eta$ is an { elementary drilled atom}. 
 
 Let $F_\CC$ denote the restriction of $F$ to the cycle $\CC$. Then $F_\CC$
 is a 3-manifold, all whose boundary components are tori. Hence  
 $\chi(F_\CC))=0$.
 By Theorem~\ref{thm-ll}, $\beta_1^{(2)} (\pi_1(F_\CC))=0$.
 
Proceeding  by induction on the number of elementary drilled atoms in $F$,
and applying Theorem~\ref{thm-fv} inductively, we conclude that 
 $\beta_1^{(2)} (G)=0$.
\end{proof}

\begin{proof}[Proof of Proposition~\ref{prop-l2}:] The first part of the argument in Section~\ref{sec-gog} allows us to modify $\GG$ to a graph where each edge group is isomorphic  to $\pi_{1}(S)$. Thus, assume without loss of generality that each edge in $\GG$ has edge group isomorphic  to $\pi_{1}(S)$.
	
Let $\GG_0$ denote a maximal subgraph of $\GG$ such that $\GG_0$ is 
homotopy equivalent to a circle. Let $F_0$ denote the restriction of $F$ to
$\GG_0$. Let $L=\pi_1(F_0)$. Let $n$ denote the number of edges of $(\GG \setminus \GG_0)$.
Then $G=\pi_1(F)$ admits a new graph of groups decomposition, where the base graph 
has one vertex $w$, and $n$ loops, with $G_w=L$, and each edge group isomorphic 
to $\pi_{1}(S)$. 

By Lemma~\ref{lem-l2},  $\beta_1^{(2)} (L)=0$.
Since $\GG_0$ is 
homotopy equivalent to a circle, edge group is of infinite index in the vertex group $L$. Hence, by Lemma~\ref{lem-induction},  $\beta_1^{(2)} (G)=0$.
\end{proof}
\subsection{Questions}
In this paper, we have only drilled simple closed curves $\sigma$ in fibers $S_x$. A similar drilling operation could be carried out even when a realization of $\sigma$ 
as a geodesic in the fiber $S$ has self-intersections. This can be done by homotoping
$\sigma$  slightly  in a product neighborhood of $S_x$ to convert  $\sigma$  into a knot
in $E$. Note that this is a non-canonical operation, as we have a choice of over-
and under-crossings at every self-intersection point. At any rate, after such a homotopy, the resulting knot can be drilled. 
We assume, however, that $\sigma$  represents a primitive element of $H=\pi_1(S)$.
We call this process \emph{generalized drilling}, and the resulting $F$ a \emph{generalized drilled bundle}. Relative hyperbolicity of the generalized drilled $G$ follows by essentially the same argument as in Theorem~\ref{thm-relhyp}.
\begin{qn}\label{qn-genlzdrill}
	Let $\Gamma = \pi_1(E)$ be hyperbolic.
	Let $F$ be a generalized drilled bundle obtained by the above generalized drilling operation applied to $E$. Is $G(:=\pi_1(F))$ cubulable?
\end{qn}

\section*{Acknowledgments} MM would like to thank Daniel Groves and Jason Manning for helpful conversations. We also thank  Daniel Groves for comments
and corrections to an earlier draft.

\newcommand{\etalchar}[1]{$^{#1}$}


\begin{thebibliography}{GHM{\etalchar{+}}23}
	
	\bibitem[Ago13]{agol-vh}
	I.~Agol.
	\newblock {The virtual Haken conjecture (With an appendix by Ian Agol, Daniel
		Groves, and Jason Manning)}.
	\newblock {\em Doc. Math. 18}, pages 1045--1087, 2013.
	
	\bibitem[Bel12]{beleg}
	Igor Belegradek.
	\newblock Rigidity and relative hyperbolicity of real hyperbolic hyperplane
	complements.
	\newblock {\em Pure Appl. Math. Q.}, 8(1):15--51, 2012.
	
	\bibitem[BF92]{BF}
	M.~Bestvina and M.~Feighn.
	\newblock A {C}ombination theorem for {N}egatively {C}urved {G}roups.
	\newblock {\em J. Diff. Geom., vol 35}, pages 85--101, 1992.
	
	\bibitem[BF96]{BFcorr}
	M.~Bestvina and M.~Feighn.
	\newblock Addendum and correction to {A} {C}ombination theorem for {N}egatively
	{C}urved {G}roups.
	\newblock {\em J. Diff. Geom., vol 43}, pages 783--788, 1996.
	
	\bibitem[BH13]{beleg-hruska}
	Igor Belegradek and G.~Christopher Hruska.
	\newblock Hyperplane arrangements in negatively curved manifolds and relative
	hyperbolicity.
	\newblock {\em Groups Geom. Dyn.}, 7(1):13--38, 2013.
	
	\bibitem[Bow06]{bowditch-guessg}
	Brian~H. Bowditch.
	\newblock Intersection numbers and the hyperbolicity of the curve complex.
	\newblock {\em J. Reine Angew. Math.}, 598:105--129, 2006.
	
	\bibitem[Bow12]{bowditch-relhyp}
	B.~H. Bowditch.
	\newblock Relatively hyperbolic groups.
	\newblock {\em Internat. J. Algebra and Computation. 22, 1250016, 66pp}, 2012.
	
	\bibitem[BW12]{bw}
	N.~Bergeron and D.~T. Wise.
	\newblock A boundary criterion for cubulation.
	\newblock {\em Amer. J. Math. 134, no. 3}, pages 843--859, 2012.
	
	\bibitem[Can96]{canary-cover}
	R.~D. Canary.
	\newblock A covering theorem for hyperbolic 3 manifolds.
	\newblock {\em Topology 35}, pages 751--778, 1996.
	
	\bibitem[DKL14]{kld-coco}
	S.~Dowdall, R.~P. Kent, and C.~Leininger.
	\newblock Pseudo-anosov subgroups of fibered 3-manifold groups.
	\newblock {\em Groups Geom. Dyn. 8 no. 4}, pages 1247--1282, 2014.
	
	\bibitem[Far98]{farb-relhyp}
	B.~Farb.
	\newblock Relatively hyperbolic groups.
	\newblock {\em Geom. Funct. Anal. 8}, pages 810--840, 1998.
	
	\bibitem[FM02]{farb-coco}
	B.~Farb and L.~Mosher.
	\newblock Convex cocompact subgroups of mapping class groups.
	\newblock {\em Geom. Topol. 6}, pages 91--152, 2002.
	
	\bibitem[FV17]{fv-hha}
	Talia Fern\'{o}s and Alain Valette.
	\newblock The {M}ayer-{V}ietoris sequence for graphs of groups, property ({T}),
	and the first {$\ell^2$}-{B}etti number.
	\newblock {\em Homology Homotopy Appl.}, 19(2):251--274, 2017.
	
	\bibitem[GHM{\etalchar{+}}23]{ghmosw}
	D.~Groves, P.~Haissinsky, J.~F. Manning, D.~Osajda, A.~Sisto, and G.~Walsh.
	\newblock Drilling hyperbolic groups.
	\newblock {\em in preparation}, 2023.
	
	\bibitem[Ham05]{hamen}
	U.~Hamenstaedt.
	\newblock {Word hyperbolic extensions of surface groups }.
	\newblock {\em preprint, arXiv:math/0505244}, 2005.
	
	\bibitem[Ham07]{hamen-guessg}
	Ursula Hamenst\"{a}dt.
	\newblock Geometry of the complex of curves and of {T}eichm\"{u}ller space.
	\newblock In {\em Handbook of {T}eichm\"{u}ller theory. {V}ol. {I}}, volume~11
	of {\em IRMA Lect. Math. Theor. Phys.}, pages 447--467. Eur. Math. Soc.,
	Z\"{u}rich, 2007.
	
	\bibitem[HW08]{hw-gafa}
	F.~Haglund and D.~T. Wise.
	\newblock Special cube complexes.
	\newblock {\em Geom. Funct. Anal. 17 no. 5}, pages 1551--1620, 2008.
	
	\bibitem[Kap01]{kapovich-book}
	M.~Kapovich.
	\newblock {Hyperbolic Manifolds and Discrete Groups}.
	\newblock {\em Progress in Mathematics, vol. 183, Birkhauser}, 2001.
	
	\bibitem[Kie20]{kielak}
	Dawid Kielak.
	\newblock Residually finite rationally solvable groups and virtual fibring.
	\newblock {\em J. Amer. Math. Soc.}, 33(2):451--486, 2020.
	
	\bibitem[KL08]{kl-coco}
	R.~P. Kent and C.~Leininger.
	\newblock Shadows of mapping class groups: capturing convex cocompactness.
	\newblock {\em Geom. Funct. Anal. 18 no. 4}, pages 1270--1325, 2008.
	
	\bibitem[KM12]{km-surf}
	J.~Kahn and V.~Markovic.
	\newblock Immersing almost geodesic surfaces in a closed hyperbolic three
	manifold.
	\newblock {\em Ann. of Math. (2) 175, no. 3}, pages 1127--1190, 2012.
	
	\bibitem[LL95]{lott-luecke}
	John Lott and Wolfgang L\"{u}ck.
	\newblock {$L^2$}-topological invariants of {$3$}-manifolds.
	\newblock {\em Invent. Math.}, 120(1):15--60, 1995.
	
	\bibitem[Mit98]{mitra-trees}
	M.~Mitra.
	\newblock Cannon-{T}hurston {M}aps for {T}rees of {H}yperbolic {M}etric
	{S}paces.
	\newblock {\em Jour. Diff. Geom.48}, pages 135--164, 1998.
	
	\bibitem[Mj20]{mahan-tight}
	Mahan Mj.
	\newblock Tight trees and model geometries of surface bundles over graphs.
	\newblock {\em Journal of the London Mathematical Society}, 102(3):1178--1222,
	jun 2020.
	
	\bibitem[MMS19]{mmscubulating}
	Jason~F. Manning, Mahan Mj, and Michah Sageev.
	\newblock Cubulating surface-by-free groups.
	\newblock {\em preprint. arxiv.1908.03545}, 2019.
	
	\bibitem[MP11]{mahan-pal}
	M.~Mj and A.~Pal.
	\newblock {Relative Hyperbolicity, Trees of Spaces and Cannon-Thurston Maps}.
	\newblock {\em Geom. Dedicata 151}, pages 59--78, 2011.
	
	\bibitem[MR08]{mahan-reeves}
	M.~Mj and L.~Reeves.
	\newblock {A Combination Theorem for Strong Relative Hyperbolicity}.
	\newblock {\em Geom. Topol. 12}, pages 1777 -- 1798, 2008.
	
	\bibitem[MR18]{mahan-rafi}
	M.~Mj and K.~Rafi.
	\newblock Algebraic ending laminations and quasiconvexity.
	\newblock {\em Algebr. Geom. Topol.}, 18(4):1883--1916, 2018.
	
	\bibitem[MS12]{mahan-sardar}
	M.~Mj and P.~Sardar.
	\newblock {A combination theorem for metric bundles}.
	\newblock {\em Geom. Funct. Anal. 22, no. 6}, pages 1636--1707, 2012.
	
	\bibitem[Sis12]{sisto2012metric}
	Alessandro Sisto.
	\newblock On metric relative hyperbolicity.
	\newblock {\em preprint, arxiv:1210.8081}, 2012.
	
	\bibitem[SS90]{scottswar}
	G.~P. Scott and G.~A. Swarup.
	\newblock Geometric finiteness of certain {K}leinian groups.
	\newblock {\em Proc. Amer. Math. Soc.}, 109(3):765--768, 1990.
	
	\bibitem[SS92]{swarsuss}
	Perry Susskind and Gadde~A. Swarup.
	\newblock Limit sets of geometrically finite hyperbolic groups.
	\newblock {\em Amer. J. Math.}, 114(2):233--250, 1992.
	
	\bibitem[Thu80]{thurstonnotes}
	W.~P. Thurston.
	\newblock The {G}eometry and {T}opology of 3-{M}anifolds.
	\newblock {\em Princeton University Notes}, 1980.
	
	\bibitem[Wis12]{wise-cbms}
	Daniel~T. Wise.
	\newblock {\em From riches to raags: 3-manifolds, right-angled {A}rtin groups,
		and cubical geometry}, volume 117 of {\em CBMS Regional Conference Series in
		Mathematics}.
	\newblock Published for the Conference Board of the Mathematical Sciences,
	Washington, DC; by the American Mathematical Society, Providence, RI, 2012.
	
	\bibitem[Wis21]{wise-hier}
	Daniel~T. Wise.
	\newblock {\em The Structure of Groups with a Quasiconvex Hierarchy}, volume
	366.
	\newblock Princeton University Press, 2021.
	
\end{thebibliography}
\end{document}